\documentclass[11pt]{article}
\usepackage[utf8]{inputenc} 
\usepackage[T1]{fontenc}    
\usepackage{lmodern}        

\usepackage{amsmath, amssymb, amsfonts, mathtools, amsthm, bm, bbm, esint}

\usepackage{graphicx}
\usepackage{caption}
\usepackage{subcaption}
\usepackage{tikz}

\usepackage{booktabs, multirow, makecell, boldline, tabularx}

\usepackage[linesnumbered,ruled,vlined]{algorithm2e}

\usepackage{enumitem}       
\usepackage{color}
\usepackage{cite}           
\usepackage{hyperref}       
\hypersetup{
    colorlinks=true,
    linkcolor=blue,
    citecolor=blue,
    urlcolor=blue,
    pagebackref=true
}

\usepackage{geometry}
\newtheorem{proposition}{Proposition}
\newtheorem{theorem}{Theorem}

\newtheorem{corollary}{Corollary}

\newtheorem{definition}{Definition}

\def\One{\mathbbm{1}} 
\newcommand{\weakconv}{ \underset{n \rightarrow \infty}{\overset{*}{\longrightarrow}}}
\newcommand{\LT}{\mathrm{L}_T}
\newcommand{\LTinv}{\mathrm{L}_{\mathrm{inv},T}}
\newcommand{\LSinv}{\mathrm{L}_{\mathrm{inv},S}}
\newcommand{\LS}{\mathrm{L}_S}

\newcommand{\MR}{\mathcal{M}(\mathbb{R})}
\newcommand{\MS}{\mathcal{M}(\mathbb{T})}

\newcommand{\MLT}{\mathcal{M}_{\LT}(\mathbb{R})}
\newcommand{\MLS}{\mathcal{M}_{\LS}(\mathbb{T})}
\newcommand{\MLSz}{\mathcal{M}_{\LS,0}(\mathbb{T})}
\newcommand{\MLTLS}{\mathcal{M}{(\LT,\LS)}}
\newcommand{\MhThS}{\mathcal{M}_{h_T,h_S}{(\LT,\LS)}}

\newcommand{\CR}{\mathcal{C}_0(\mathbb{R})}
\newcommand{\CT}{\mathcal{C}(\mathbb{T})}
\newcommand{\CLT}{\mathcal{C}_{\LT}(\mathbb{R})}
\newcommand{\CLS}{\mathcal{C}_{\LS}(\mathbb{T})}
\newcommand{\CLSz}{\mathcal{C}_{\LS,0}(\mathbb{T})}
\newcommand{\CLTLS}{\mathcal{C}{(\LT,\LS)}}

\newcommand{\R}{ \mathbb{R}}
\newcommand{\T}{ \mathbb{T}}
\newcommand{\Z}{ \mathbb{Z}}
\newcommand{\N}{ \mathbb{N}}
\newcommand{\Per}{\mathrm{Per}}
\newcommand{\Lop}{{\rm L}}

\newcommand{\Dop}{{\rm D}}

\newcommand{\Span}{\mathrm{Span}}

\newcommand{\Sp}{ \mathcal{S}}

\title{Variational Seasonal-Trend Decomposition with \\  Sparse Continuous-Domain Regularization}
\author{Julien Fageot}

\begin{document}

\maketitle

\date{} 
\begin{abstract}
We consider the inverse problem of recovering a continuous-domain function from a finite number of noisy linear measurements. The unknown signal is modeled as the sum of a slowly varying trend and a periodic or quasi-periodic seasonal component. We formulate a variational framework for their joint recovery by introducing convex regularizations based on generalized total variation, which promote sparsity in spline-like representations. Our analysis is conducted in an infinite-dimensional setting and leads to a representer theorem showing that minimizers are splines in both components. To make the approach numerically feasible, we introduce a family of discrete approximations and prove their convergence to the original problem in the sense of $\Gamma$-convergence. This further ensures the uniform convergence of discrete solutions to their continuous counterparts. The proposed framework offers a principled approach to seasonal-trend decomposition in the presence of noise and limited measurements, with theoretical guarantees on both representation and discretization.
\end{abstract}

\section{Introduction}\label{sec:intro}

We develop variational methods for reconstructing a continuous-domain function 
$f_0 : \R \rightarrow \R$, assumed to be the sum $f_0 = f_{0,T} + f_{0,S}$ of two components: a trend $f_{0,T}$ and a seasonal part $f_{0,S}$. The seasonal component is assumed to be periodic with known period, which we normalize to $1$ without loss of generality. We are given access to finite-dimensional, possibly corrupted linear measurements in the form of an observation vector $\bm{y} \in \R^L$. These measurements correspond to a linear sensing vector
\begin{equation}
    \bm{\Phi}(f_0) = (\phi_1(f_0), \ldots , \phi_L(f_0)),
\end{equation}
where each $\phi_\ell$ is a continuous linear functional acting on $f_0$. Our objective is to recover both the trend and seasonal components jointly, using a continuous-domain variational framework based on sparsity-promoting regularization.

\subsection{Seasonal-Trend Sparse Inverse Problem} \label{sec:inverseproblem}

We formulate the recovery of $(f_{0,T},f_{0,S})$ as a functional inverse problem. The seasonal component $f_{0,S}$ is modeled as a function $f_{0,S} : \mathbb{T} \rightarrow \mathbb{R}$ defined on the circle $\mathbb{T}$, reflecting its periodicity of period $1$. The reconstruction is posed as the solution to the optimization problem
\begin{equation} \label{eq:firstoptiproblem}
    \underset{(f_T,f_S)}{\arg \min} \quad \mathcal{E}\left( \Phi(f_T+f_S) \right) + \lambda_T   \lVert \LT f_T \rVert_{\MR}  + \lambda_S  \lVert \LS f_S \rVert_{\MS},
\end{equation}
where:
\begin{itemize}
    \item The data-fidelity term $\mathcal{E}\left( \Phi(f_T+f_S) \right)$ encourages the synthesized measurements
    \begin{equation}
        \Phi(f_T+f_S) = (\phi_1(f_T+f_S), \ldots , \phi_L (f_T+f_S)) \in \R^L
    \end{equation}
    to match the observations $\bm{y}$. A common choice for $\mathcal{E}$ is the quadratic loss
    \begin{equation} \label{eq:quadracost}
        \mathcal{E}\left( \bm{z} \right) = \| \bm{y} -  \bm{z} \|^2 = \sum_{\ell=1}^L (y_\ell -  z_\ell )^2,
    \end{equation}
    which corresponds to a Gaussian white noise model in a Bayesian interpretation.

    \item The regularization term
    \begin{equation}
        \mathcal{R}(f_T,f_S) =  \lambda_T \lVert \LT f_T \rVert_{\MR} +  \lambda_S \lVert \LS f_S \rVert_{\MS}
    \end{equation}
    imposes structure on the trend and seasonal components via the differential operators $\LT$ and $\LS$. The norms $\lVert \cdot \rVert_{\MR}$ and $\lVert \cdot \rVert_{\MS}$ are continuous-domain generalizations of the $\ell_1$ norm, promoting sparsity in the sense of Radon measures. The former promotes piecewise-smooth or spline-like structure in $f_T$, while the latter enforces periodicity in $f_S$~\cite{Fisher1975}.

    \item The parameters $\lambda_T > 0$ and $\lambda_S > 0$ control the trade-off between data fidelity, trend regularization, and seasonal regularization.
\end{itemize}

A solution $(\tilde{f}_T, \tilde{f}_S)$ to \eqref{eq:firstoptiproblem} is expected to approximate the true decomposition $(f_{0,T}, f_{0,S})$, and hence $\tilde{f} = \tilde{f}_T + \tilde{f}_S$ provides a faithful reconstruction of $f_0 = f_{0,T} + f_{0,S}$. To the best of our knowledge, total-variation regularization has not been previously used for seasonal-trend decomposition in the continuous domain. Yet, as discussed in Section~\ref{sec:related}, there is strong theoretical and empirical motivation for such an approach.

\subsection{Related Works} \label{sec:related}

Our work aims at bridging functional inverse problems based on total-variation (TV) regularization with the reconstruction and decomposition of seasonal-trend signals. \\

\textit{TV-based regularization: from Dirac impulses to splines.}
TV regularization was initially introduced as a prior for recovering Dirac impulses from linear observations. In particular, optimization problems of the form $\min_{w \in \mathcal{C}} \| w \|_{\MR}$, where $\mathcal{C}$ is a weak*-compact convex subset of the space $\MR$ of Radon measures, admit Dirac stream solutions.

The roots of TV-based convex optimization can be traced back to the pioneering works~\cite{beurling1938integrales,Zuhovickii1948,Fisher1975}. 
Inspired by analogous results in discrete compressed sensing~\cite{Donoho2006,Candes2006sparse,Foucart2013mathematical} and sparse statistical learning~\cite{tibshirani1996regression,chen2001atomic,Elad10}, several authors have explored convex optimization techniques for the recovery of Dirac measures~\cite{Bredies2013inverse,deCastro2012exact,candes2014towards,Duval2015exact,Azais2015Spike,denoyelle2019sliding,debarre2022part1}.

More recently, these frameworks have been extended to spline recovery~\cite{unser2017splines,gupta2018continuous,flinth2019exact,fageot2020tv,simeoni2021functional}, by considering generalized TV-type regularizations of the form
\begin{equation}
    f \mapsto \| \Lop f \|_{\MR},
\end{equation}
where $\Lop$ is a pseudo-differential operator.
Both non-periodic~\cite{unser2017splines} and periodic~\cite{fageot2020tv} settings have been considered, and these constructions serve as the foundation of our seasonal-trend framework.

As detailed in Section~\ref{sec:contributions}, our methodology follows the standard steps involved in TV-based continuous-domain inverse problems: the specification of the native space of the problem along with its measurement space~\cite{unser2017splines,unser2019native,fageot2020tv}, the derivation of a representer theorem characterizing the solutions~\cite{Fisher1975,unser2017splines,gupta2018continuous,boyer2019representer,bredies2020sparsity,debarre2022sparsest}, and optention of grid-based convergence results leading to the design of grid-based algorithms for practical reconstruction~\cite{duval2017sparseI,Debarre2019,simeoni2021functional,debarre2022part2,Flinth2025Grid, guillemet2025convergence}. \\

\textit{Multi-component optimization.}
Multi-component approaches aim at recovering a signal $s = s_1 + s_2$ expressed as the superposition of two (or more) underlying components with distinct structural properties. 
In the discrete setting, such decompositions are related to source separation~\cite[Chapter 11]{eldar2012compressed} and have been explored in compressed sensing under the frameworks of sparse representations in unions of dictionaries~\cite{gribonval2003sparse,lu2008theory,elad2010role}, or sparse + smooth recovery~\cite{gholami2013balanced,grasmair2020adaptive,debarnot2020learning}.

In the continuous-domain setting, TV-based multi-component reconstruction was first considered in~\cite{debarre2019hybrid} for hybrid spline decompositions, and has been applied to supervised learning~\cite{aziznejad2021multikernel} and sparse + smooth signal recovery~\cite{debarre2021continuous}.
Theoretical foundations of such approaches—including existence, uniqueness, and structural properties of solutions—have been developed in~\cite{unser2022convex}.

To the best of our knowledge, this is the first attempt to extend such a multi-component TV-regularized framework to seasonal-trend signal decomposition. \\

\textit{Seasonal-trend decomposition.}
A key novelty of our work is to bring a seasonal-trend perspective into the realm of continuous-domain inverse problems regularized by total variation.

\subsection{Contributions and Outline} \label{sec:contributions}

We develop the mathematical analysis of functional inverse problems for sparsity-promoting seasonal-trend reconstruction. A key feature of our approach is its generality, with particular attention paid to covering broad classes of regularization operators $(\LT,\LS)$, measurement functionals $\phi_\ell$, and convex data-fidelity terms $\mathcal{E}$.

\begin{itemize}
    \item \textbf{Functional-analytic foundations:} We introduce the general framework required to properly define the variational problem~\eqref{eq:firstoptiproblem}. This involves specifying the \textit{native Banach space} under which the optimization is conducted, as well as its joint \textit{measurement space}, where the measurement functionals $\phi_\ell$ reside. We uncover the inherent direct sum structure of the problem, ensuring the identifiability of the seasonal-trend model.

    \item \textbf{Seasonal-trend representer theorem:} We establish a representer theorem, demonstrating that the variational problem~\eqref{eq:firstoptiproblem} admits composite sparse spline solutions for both the seasonal and trend components.

    \item \textbf{Quadratic representer theorem:} We compare the use of total-variation regularization norms with quadratic ones, which are more classical and linked to kernel methods. From this, we derive a quadratic representer theorem for seasonal-trend decomposition, showing that the unique solution of the quadratic variational problem strongly couples the seasonal and trend components, thereby limiting the utility of kernel methods.

    \item \textbf{Theoretical analysis of grid-based approximation:} To solve the optimization task~\eqref{eq:firstoptiproblem} efficiently, we introduce a grid-based approximation scheme and perform a convergence analysis. We demonstrate the $\Gamma$-convergence of grid-based schemes and, under mild assumptions, deduce the pointwise and uniform convergence of the grid-based solutions towards the continuous-domain ones, for both the trend and seasonal components.

\end{itemize}

The paper is organized as follows:
After introducing the relevant mathematical concepts in Section~\ref{sec:prelim}, we present the seasonal-trend native spaces for optimization problems of the form~\eqref{eq:firstoptiproblem} in Section~\ref{sec:native}. In Section~\ref{sec:sensing}, we characterize the sensing functionals $\phi_\ell$ used in the observation process. The Representer Theorem for~\eqref{eq:firstoptiproblem}, which specifies the form of its extreme point solutions, is presented in Section~\ref{sec:RT}. Finally, we study the convergence properties of grid-based methods for approximating solutions to~\eqref{eq:firstoptiproblem} in Section~\ref{sec:gridbased}.

\section{Functional Analytic Preliminaries}\label{sec:prelim}

    \subsection{Tempered Generalized Functions} \label{sec:schwartz}

    We begin by introducing the Schwartz space of rapidly decaying smooth functions and its dual, along with their periodic counterparts. For a concise introduction to the subject, we refer the reader to~\cite{Simon2003distributions}. It is important to note that these spaces will primarily serve as the space of "nice" test functions and the space of all possible linear functionals encountered in this work.\\

    Let $\Dop^N$ denote the $N$th derivative. The Schwartz space $\Sp(\R)$ consists of infinitely smooth, rapidly decaying functions, endowed with the Fréchet topology associated with the family of semi-norms
    $$ \varphi \in \Sp(\R) \mapsto \sup_{t \in \R} |t^M \Dop^N \varphi (t) | \quad \text{for} \quad M, N \geq 0. $$

    The topological dual of $\Sp(\R)$ is the space of tempered generalized functions, denoted by $\Sp'(\R)$, and it is endowed with the weak*-topology. The duality pairing between $f \in \Sp'(\R)$ and $\varphi \in \Sp(\R)$ is denoted by $\langle f, \varphi \rangle_{\Sp'(\R)\times \Sp(\R)}$. Any generalized function encountered in this work is identified as a tempered generalized function through its action on test functions $\varphi \in \Sp(\R)$. The function spaces considered in this work are all embedded in $\Sp'(\R)$.

    We also define the space of smooth $1$-periodic functions $\Sp(\T)$, where $\T = [0,1]$ is the $1$-dimensional torus, with the endpoints $0$ and $1$ identified. The space $\Sp(\T)$ is a Fréchet space, equipped with the topology associated with the family of semi-norms
    $$ \varphi \in \Sp(\T) \mapsto \sup_{t \in \R} | \Dop^N \varphi (t) | \quad \text{for} \quad N \geq 0. $$

    The topological dual of $\Sp(\T)$ is the space $\Sp'(\T)$ of $1$-periodic generalized functions, which is also endowed with the weak*-topology. The duality pairing between $f \in \Sp'(\T)$ and $\varphi \in \Sp(\T)$ is denoted by $\langle f, \varphi \rangle_{\Sp'(\T)\times \Sp(\T)}$.

    We have the topological embeddings $\Sp(\T) \subseteq \Sp'(\T) \subseteq \Sp'(\R)$. Thus, a periodic generalized function $f \in \Sp'(\T)$ can be viewed as a tempered generalized function by its action on test functions in $\Sp(\R)$, rather than just on $\Sp(\T)$. Furthermore, for $f \in \Sp'(\T) \subseteq \Sp'(\R)$ and $\varphi \in \Sp(\R)$, we have
    \begin{equation}
        \langle f, \varphi \rangle_{\Sp'(\R) \times \Sp(\R)} = \langle f, \Per\{\varphi\} \rangle_{\Sp'(\T) \times \Sp(\T)},
    \end{equation}
    where $\Per: \Sp(\R) \to \Sp(\T)$ is the periodization operator defined by
    $$ \Per \{\varphi\} = \sum_{k \in \Z} \varphi(\cdot - k). $$

    \subsection{Linear Shift-Invariant Operators} \label{sec:LSI}
    
    Let $\Lop$ be a linear, shift-invariant (LSI), and continuous operator from $\Sp'(\R)$ to itself. Then, $\Lop$ is a convolution operator of the form
    \begin{equation} \label{eq:convolution}
        \Lop f = h * f
    \end{equation}
    where the impulse response $h$ of $\Lop$ belongs to the space $\mathcal{O}_C'(\R)$ of rapidly decreasing generalized functions~\cite{fischer2015duality}. This space includes the Dirac delta $\delta$, for which $\Lop = \mathrm{Id}$ is the identity operator, and its derivatives $\Dop^N \delta = \delta^{(N)}$ of any order $N \geq 1$, for which $\Lop = \Dop^N$. More generally, any compactly supported generalized function is in $\mathcal{O}_C'(\R)$, which also includes rapidly decaying smooth functions such as $h(t) = \exp(- t^2/2)$.

    The space $\mathcal{O}_C'(\R)$ is the Fourier transform of the space $\mathcal{O}_M(\R)$ of infinitely smooth and slowly growing functions\footnote{A function $g : \R \to \R$ is in $\mathcal{O}_M(\R)$ if it is infinitely smooth, and if every derivative $\Dop^N g$ grows no faster than a polynomial.}. Both $\mathcal{O}_C'(\R)$ and $\mathcal{O}_M(\R)$ have been studied in~\cite[Chapter VII-5]{Schwartz1966distributions} and correspond to convolution and multiplication operators from $\Sp'(\R)$ to itself\footnote{This explains the subscripts $C$ (for convolution) and $M$ (for multiplication) in the definition of these spaces.}.

    The frequency response $\widehat{L}$ of a continuous LSI operator $\Lop$ is the Fourier transform $\widehat{L} = \widehat{h}$ of its impulse response $h$. The frequency response $\widehat{L}$ belongs to $\mathcal{O}_M(\R)$, and satisfies $\widehat{\Lop f} = \widehat{L} \widehat{f}$ for any $f \in \Sp'(\R)$.

\subsection{Periodic and Non-Periodic Radon Measures} \label{sec:radon}

The space of Radon measures is denoted by $\MR$. It is the topological dual of the space $\CR$ of continuous functions vanishing at infinity, endowed with the supremum norm~\cite{gray1984shaping}. We denote by $\langle w , g \rangle_{\MR, \CR} = \int_{\R} \varphi(t) \, \mathrm{d} w (t)$ the duality product between a function $g \in \CR$ and a Radon measure $w \in \MR$.

The space of Radon measures is a Banach space for the total variation norm defined by
\begin{equation}
    \| w \|_{\MR} = \sup_{\varphi \in \CR, \ \| g \|_\infty = 1} \langle w , g \rangle_{\MR \times \CR}.
\end{equation}
Radon measures include integrable functions $f \in L_1(\R)$, for which $\| f \|_{\MR} = \|f\|_{L_1(\R)} = \int_{\R} |f(t)| \, \mathrm{d} t$, and shifted Dirac impulses $\delta(\cdot - x_0)$ for any $x_0 \in \R$. For a sparse Radon measure $w = \sum_{k=1}^K a_k \delta(\cdot - t_k)$ with weights $\bm{a}= (a_1,\ldots , a_K)\in \R^K$ and distinct knots $t_k \in \R$, we have
\begin{equation}
    \left\lVert \sum_{k=1}^K a_k \delta(\cdot - t_k) \right\rVert_{\MR} = \sum_{k=1}^K |a_k| = \lVert \bm{a} \rVert_{1}.
\end{equation}
In this sense, the total variation norm is the functional extension of the discrete $\ell_1$ norm.

We also consider the space $\MS$ of \emph{periodic} Radon measures. It is the topological dual of the space $\CT$ of periodic continuous functions for the supremum norm. The space $\MS$ includes periodic integrable functions $f \in L_1(\T)$ and any shifted version of the Dirac stream $\mathrm{III} = \sum_{k \in \Z} \delta (\cdot-k) = \Per \{\delta\}$. The total variation norm of a periodic Radon measure is denoted by
\begin{equation}
    \| w \|_{\MS} = \sup_{g \in \CT, \ \|g\|_\infty = 1} \langle w , g \rangle_{\MS \times \CT},
\end{equation}
where the duality product stands for $\langle w , g \rangle_{\MS \times \CT} = \int_\T g(t) \, \mathrm{d}w(t)$. A sparse periodic Radon measure $w = \sum_{k=1}^K a_k \mathrm{III}(\cdot - t_k)$ with weights $\bm{a}= (a_1,\ldots , a_K)\in \R^K$ and distinct knots $t_k \in \T$ satisfies
\begin{equation}
    \left\lVert  \sum_{k=1}^K a_k \mathrm{III}(\cdot - t_k) \right\rVert_{\MS} = \sum_{k=1}^K |a_k| = \lVert \bm{a} \rVert_{1}.
\end{equation}

Finally, we have the topological embeddings between Schwartz and Radon spaces:
\begin{align}
    \Sp(\R) \subseteq  \MR \subseteq \Sp'(\R) \quad \text{and} \quad 
    \Sp(\T) \subseteq  \MS \subseteq \Sp'(\T).
\end{align}

\section{Native Spaces for Seasonal-Trend Decomposition}\label{sec:native}

We introduce the functional-analytic framework for the specification of the native space of the variational problem~\eqref{eq:firstoptiproblem}. The optimization is made over pairs of functions $(f_T,f_S)$, respectively the trend and seasonal components. 
Our methods are suitable for signals resulting from the superposition of periodic and acyclic processes depending on continuous-domain variables.

\subsection{Trend Operators, Splines, and Native Spaces} \label{sec:trendop}

\paragraph{Trend Operators.}
We consider LSI and continuous operators $\LT : \Sp'(\R) \rightarrow \Sp'(\R)$. The frequency response of $\LT$ is denoted by $\widehat{L}_T$ and is in $\mathcal{O}_M(\R)$ (see Section~\ref{sec:LSI}).

\begin{definition}
\label{def:trendop}
Let $\LT : \Sp'(\R) \rightarrow \Sp'(\R)$ be an LSI and continuous operator. Then, $\LT$ is \emph{trend-admissible of order $N_T \geq 0$} if it can be written as 
\begin{equation} \label{eq:LTdecomposed}
    \LT = \Dop^{N_T} \LTinv
\end{equation}
where $\LTinv : \Sp'(\R) \rightarrow \Sp'(\R)$ is LSI, continuous, and invertible. 
\end{definition}

The invertibility of $\LTinv$ means that there exists an LSI and continuous operator $\LTinv^{-1} : \Sp'(\R) \rightarrow \Sp'(\R)$ such that $\LTinv \LTinv^{-1} f =  \LTinv^{-1} \LTinv f = f$ for any $f \in \Sp'(\R)$. This implies that the frequency response $\omega \mapsto \widehat{L}_{\mathrm{inv},T}(\omega)$ does not vanish on $\R$. The frequency response of the inverse is then the function $\widehat{L}_{\mathrm{inv},T}^{-1} \in \mathcal{O}_M(\R)$. The invertibility condition therefore imposes that $\widehat{L}_{\mathrm{inv},T}^{-1}$ is slowly growing, hence $\widehat{L}_{\mathrm{inv},T}$ cannot vanish faster than any polynomial. This excludes, for instance, the frequency response $\widehat{L}_{\mathrm{inv},T}(\omega) = \exp(- \omega^2/2)$ corresponding to an operator $\Lop f = h * g$ with $h(t) = \frac{1}{\sqrt{2\pi}} \exp( - t^2 / 2)$. 

The frequency response of $\LT$ is $\widehat{L}_T(\omega) = (\mathrm{i}\omega)^{N_T} \widehat{L}_{\mathrm{inv},T}(\omega)$. $\LT$ is invertible if and only if its trend-admissibility order is ${N_T} = 0$. 
Invertible trend-admissible operators include the Sobolev operators $\LT = (\mathrm{Id} - \Delta)^{\gamma/2}$ with $\gamma \in \R$, whose frequency response is $\widehat{L}_T (\omega) = (1 + \omega^2)^{\gamma / 2}$. We also consider non-invertible trend-admissible operators to include the derivative operators $\Dop^{N_T}$. The null space of a trend-admissible operator of order ${N_T} \geq 1$ is the space $\mathcal{P}_{N_T}(\R)$ of polynomials up to degree $({N_T} - 1)$, whose dimension is $N_T$.

The composition of two trend-admissible operators with orders $N_{T,1}, N_{T,2} \geq 0$ is trend-admissible with order $N_{T,1} + N_{T,2}$, and its frequency response is the product of the two underlying frequency responses.

\subsection{Seasonal Operators, Splines, and Native Spaces} \label{sec:seasonoop}

\paragraph{Seasonal Operators.}
    We consider LSI and continuous \textit{periodic} operators \( \LS : \Sp'(\T) \rightarrow \Sp'(\T) \) acting on periodic generalized functions. The complex exponentials \( e_n(t) = \mathrm{e}^{ 2 \mathrm{i} \pi n t} \), \( n \in \mathbb{Z} \), are the eigenfunctions of such operators and we have that \( \LS e_n = \widehat{L}_S[n] e_n \) for some \( \widehat{L}_S[n] \in \mathbb{C} \). We call \( \left(\widehat{L}_S[n] \right)_{n\in \mathbb{Z}} \) the \textit{frequency sequence} of \( \LS \).

    \begin{definition}
    \label{def:trendop}
    Let \( \LS : \Sp'(\T) \rightarrow \Sp'(\T) \) be an LSI and continuous periodic operator. Then, \( \LS \) is \emph{seasonal-admissible of order \( N_S \geq 0 \)} if it can be written as 
    \begin{equation}
        \LS = \Dop^{N_S} \LSinv
    \end{equation}
    where \( \LSinv : \Sp'(\T) \rightarrow \Sp'(\T) \) is LSI, continuous, and invertible. 
    \end{definition}
    
    The invertibility of \( \LSinv \) means that there exists an LSI and continuous periodic operator \( \LSinv^{-1} : \Sp'(\T) \rightarrow \Sp'(\T) \) such that \( \LSinv \LSinv^{-1} f =  \LSinv^{-1} \LSinv f = f \) for any \( f \in \Sp'(\T) \). This implies that \( \widehat{L}_S[n] \neq 0 \) for \( n \in \mathbb{Z} \) and the frequency sequence of the inverse of \( \LSinv^{-1} \) is then the sequence \( \left( \frac{1}{\widehat{L}_S[n]} \right)_{n\in \mathbb{Z}} \). 
    
    A seasonal-admissible operator \( \LS \) is invertible if and only if its seasonal-admissibility order \( N_S = 0 \). The periodic Sobolev operators \( \LS = (\mathrm{Id} - \Delta)^{\gamma/2} \) with \( \gamma \in \mathbb{R} \) are invertible and their frequency sequence is \( \widehat{L}_S[n] = (1 + 4 \pi^2 n^2)^{\gamma / 2} \).
    
    Non-invertible seasonal-admissible operators include the derivative operators \( \Dop^{N_S} \), whose frequency sequence is \( ((2\pi n)^{N_S})_{n \in \mathbb{Z}} \). The null space of a non-invertible seasonal-admissible operator is of dimension 1 and made of constant functions. This is a significant distinction with the case of trend operators.
    
    The composition of two seasonal-admissible operators with order \( N_1, N_2 \geq 0 \) is seasonal-admissible with order \( N_1 + N_2 \), and its frequency sequence is the pointwise product of the two frequency sequences of the underlying operators.
    
    A linear and shift-invariant operator \( \Lop : \Sp'(\R) \rightarrow \Sp'(\R) \) specifies a periodic operator by restricting its domain to \( \Sp'(\T) \). Moreover, we have the following result.
    
    \begin{proposition}
    Let \( \Lop \) be a trend-admissible operator with trend-admissibility order \( N_T \geq 0 \). Then, the following facts hold:
    \begin{itemize}
        \item The operator \( \Lop \) is seasonal-admissible when seen as a periodic operator. 
        
        \item Its seasonal-admissibility order equals its trend-admissibility order.
        
        \item The periodic null space of \( \Lop \) is trivial if \( N_T = 0 \) and made of constant functions otherwise. The dimension of the periodic null space is therefore \( \min (1, N_T) \). 
        
        \item In particular, \( \Lop \) is invertible as an operator on \( \Sp'(\R) \), if and only if it is invertible as a periodic operator on \( \Sp'(\T) \).
    \end{itemize}
    \end{proposition}
    
    \begin{proof}
    By definition, we have that \( \Lop = \Dop^{N_T} \Lop_{\mathrm{inv}} \) for some invertible operator \( \Lop_{\mathrm{inv}} \). Then, \( \Lop_{\mathrm{inv}} : \Sp'(\R) \rightarrow \Sp'(\R) \) is LSI, hence its frequency response \( \widehat{L}_{\mathrm{inv}} \in \mathcal{O}_M(\R) \) is a smooth function bounded by some polynomials~\cite{Schwartz1966distributions}. Since \( \Lop_{\mathrm{inv}} \) is invertible, the frequency response \( 1 / \widehat{L}_{\mathrm{inv}} \) of its inverse is also a smooth function bounded by some polynomial. Hence, there exist constants \( A, B > 0 \) and \( p > 0 \) such that
        \begin{equation} \label{eq:somebounds}
            \frac{A}{1+|\omega|^p} \leq \widehat{L}_{\mathrm{inv}}(\omega) \leq B (1+|\omega|^p)
        \end{equation}
    for any \( \omega \in \mathbb{R} \). \\
    
    Since \( \Lop \) is a convolution, it commutes with shift operations. Hence, for any \( f \in \Sp'(\T) \), we have that \( \Lop \{f\} (\cdot - 1) = \Lop \{ f (\cdot - 1)\} = \Lop \{f\} \), which is therefore \( 1 \)-periodic. This shows that \( \Lop : \Sp'(\T) \rightarrow \Sp'(\T) \) is a periodic operator. The same holds for \( \Lop_{\mathrm{inv}} \). 
    
    The frequency sequence of \( \Lop_{\mathrm{inv}} \) (seen as a periodic operator) is connected to its frequency response (as an operator from \( \Sp'(\R) \) to itself) via the relation \( \widehat{L}_{\mathrm{inv}}[n] = \widehat{L}_{\mathrm{inv}}(2\pi\omega) \). Hence, due to \eqref{eq:somebounds}, we have that
        \begin{equation}
            \frac{A'}{1+|n|^p} \leq \widehat{L}_{\mathrm{inv}}[n] \leq B' (1+|n|^p)
        \end{equation}
        for some constants \( A', B' > 0 \) and \( p > 0 \). This implies that \( \Lop_{\mathrm{inv}} \) is invertible as a periodic operator~\cite{fageot2020tv}. 
        
        The equality of the trend-admissibility and the seasonal admissibility orders of \( \Lop \) is obvious. 
        
        The periodic null space of \( \Lop \) is the same as the one of \( \Dop^{N_T} \), which is \( \{0\} \) if \( N_T = 0 \) and \( \mathrm{Span} \{1\} \) otherwise. 
        
        Finally, the invertibility condition is obvious using that \( \L \) is invertible as a trend operator if and only if \( N_T = 0 \), and the latter condition also holds for \( \Lop \) as a seasonal operator. 
    \end{proof}
 
\paragraph{Periodic \( \LS \)-Splines.}
   We introduce periodic splines, whose concept slightly differs from the classical splines. In particular, as soon as the operator \( \LS \) has a null space, there is no function \( f \) such that \( \LS f = \mathrm{III} \)\footnote{For instance, for any \( f \in \Sp'(\T) \), we have that \( \langle \Dop f , 1 \rangle_{\Sp'(\T)\times \Sp(\T)} = 0 \neq \langle \mathrm{III} , 1 \rangle_{\Sp'(\T)\times \Sp(\T)} = 1 \), hence \( \Dop f \neq \mathrm{III} \). See~\cite[Section 2.2]{fageot2020tv} for more details.}. 
   We follow~\cite[Definition 4]{fageot2020tv} for \emph{periodic} Green's functions.  
   
    \begin{definition}
    \label{def:trendsplines}
    Let \( \LS \) be a seasonal-admissible operator with order \( N_S \geq 0 \). 
    \begin{itemize}
    
        \item When \( N_S = 0 \), we say that \( \rho_{\LS} \) is a \emph{periodic Green's function of \( \LS \)} if \( \LS \psi_{\LS} = \mathrm{III} \). 
        
        \item For \( N_S \geq 1 \), \( \rho_{\LS} \) is a periodic Green's function of \( \LS \) if \( \LS \psi_{\LS} = \mathrm{III} - 1 \). 
        
        \item For any \( N_S \geq 0 \), we say that \( f \) is a \emph{periodic \( \LS \)-spline} if 
    \begin{equation} \label{eq:periodicinnovation}
        \LS f = \sum_{k=1}^K a_k \mathrm{III}(\cdot - t_k)
    \end{equation}
    where \( K \geq 0 \), the \( a_k \in \mathbb{R} \backslash \{0\} \) are the \emph{weights}, and the  \( t_k \in \T \) are the distinct \emph{knots} of \( f \).
    \end{itemize}
    \end{definition}

 Assume that $\LS$ is invertible. Then, the generic form of a periodic $\LS$-spline is
\begin{equation}
    f = \sum_{k=1}^K a_k \rho_{\LS} (\cdot - t_k).
\end{equation}
For seasonal-admissible operators of order $N_S \geq 1$, we first observe that the Green's function itself $\rho_{\LS}$ is \textit{not} a periodic $\LS$-spline since $\LS \rho_{\LS} = \mathrm{III} - 1$ is not a Dirac stream as in \eqref{eq:periodicinnovation}. In that case, the generic form of a periodic $\LS$-spline is
\begin{equation}
    f = \sum_{k=1}^K a_k \rho_{\LS}(\cdot - t_k) + \alpha,
\end{equation}
where $\sum_{k=1}^K a_k = 0$ and $\alpha \in \R$~\cite[Proposition 3]{fageot2020tv}. The condition $\sum_{k=1}^K a_k = 0$ ensures that $f$ is indeed a periodic $\LS$-spline, since
\begin{equation}
    \Lop_S f = \sum_{k=1}^K a_k  \Lop_S \{\rho_{\LS}\}(\cdot - t_k) 
    = \sum_{k=1}^K a_k \mathrm{III} (\cdot - t_k) - \sum_{k=1}^K a_k = \sum_{k=1}^K a_k \mathrm{III} (\cdot - t_k).
\end{equation}

\paragraph{Seasonal Native Spaces.}
We fix a seasonal-admissible operator $\LS$ and define the space
\begin{equation}
   \MLS = \{ f \in \Sp'(\T), \quad \LS f \in \MS \}.
\end{equation}
Such spaces generalize the space of periodic Radon measures and have been studied in~\cite{fageot2020tv}. Seasonal native spaces are simpler to understand than their trend counterparts. We summarize the important facts specifying the Banach space structure of these spaces. More details can be found in~\cite[Section 3.2]{fageot2020tv}.

\begin{itemize}
    \item If $\LS$ is invertible, then $\MLS = \LS^{-1} \{ \MS \}$ easily inherits the Banach space structure of $\MR$ for the norm $f \mapsto \| f \|_{\MLS} = \| \LT f \|_{\MR}$.

    \item If $\LS = \Dop^{N_S} \Lop_{S,\mathrm{inv}}$ with $N_S \geq 1$, then $f \mapsto \| \LS f \|_{\MS}$ only specifies a semi-norm, and the null space of $\LS$ is the 1-dimensional space of constant functions. An element $f$ of $\MLS$ is then characterized by $(\LS f , \langle f , 1 \rangle_{\Sp'(\T) \times \Sp(\T)}) \in \MS \times \R$. Then, $\MLS$ is a Banach space for the norm
    \begin{equation}
        \| f \|_{\MLS} = \| \LS f \|_{\MS} + |\langle f , 1 \rangle_{\Sp'(\T) \times \Sp(\T)}|.
    \end{equation}

    \item We also consider the function space
    \begin{equation}
       \MLSz = \left\{ f \in \Sp'(\T), \quad \LS f \in \MS \quad \text{and} \quad \langle f , 1 \rangle_{\Sp'(\T) \times \Sp(\T)} = 0 \right\}
    \end{equation}
    of elements of $\MLS$ with zero mean. It is also a Banach space for the norm $\| \cdot \|_{\MLS}$.
\end{itemize}

The spaces $\MLS$ and $\MLSz$ are subspaces of $\Sp'(\T)$, which is itself embedded in $\Sp'(\R)$. We therefore identify a function $f \in \MLS$ with its corresponding periodic function defined over the real line $\R$.

\subsection{The Seasonal-Trend Native Space} \label{sec:STspace}

We fix a pair of trend- and seasonal-admissible operators $(\LT,\LS)$. We first consider the function space
\begin{equation} \label{eq:MLTplusMLS}
    \MLT + \MLS = \left\{ f = f_T + f_S, \quad (f_T,f_S) \in \MLT \times \MLS \right\}.
\end{equation}
The sum in~\eqref{eq:MLTplusMLS} is not always direct, as we clarify in the next proposition.

\begin{proposition} \label{prop:directsumdecomposition}
Let $\LT$ and $\LS$ be trend-admissible and seasonal-admissible operators, with respective orders $N_T, N_S \geq 0$.
\begin{itemize}
    \item Assume that $N_T = 0$. Then, we have the direct sum
    \begin{equation} \label{eq:directwithno0mean}
       \MLT + \MLS = \MLT \oplus \MLS.
    \end{equation}
    
    \item Assume that $N_T \geq 1$. Then, we have the direct sum
    \begin{equation} \label{eq:directwith0mean}
        \MLT + \MLS = \MLT \oplus \MLSz.
    \end{equation}
\end{itemize}
\end{proposition}

\begin{proof}
We prove the relation \eqref{eq:directwith0mean} (the proof of \eqref{eq:directwithno0mean} is similar and simpler). First of all, $\MLT + \MLS = \MLT + \MLSz$ since the constant function $1$ is in $\MLS$ as soon as $N_T \geq 1$, and $\MLS  =  \MLSz \oplus \Span\{1\}$. It then suffices to show that the sum between $\MLS$ and $\MLSz$ is direct.

Let $f \in \MLT \cap \MLSz$. The result follows if we prove that $f = 0$. First, $f \in \MLSz \subseteq \Sp'(\T)$ is a 1-periodic generalized function, hence $g = \LT f \in \Sp'(\T)$. Moreover, $g \in \MR$ by assumption. Hence, $g \in \MR \cap \Sp'(\T)$.

Let $\varphi \in \mathcal{S}(\R)$. For $n \geq 0$, we set $\varphi_n = \sum_{k=0}^n \varphi(\cdot - k)$. The periodicity of $g$ implies that
\begin{equation}
    \langle g , \varphi_n \rangle_{\Sp'(\R)\times\Sp(\R)} = \sum_{k=0}^n \langle g, \varphi(\cdot - k ) \rangle_{\Sp'(\R)\times\Sp(\R)} = (n+1) \langle g, \varphi \rangle_{\Sp'(\R)\times\Sp(\R)}.
\end{equation}
Then, using that $g \in \MR$ and $\varphi \in \CR$, we deduce that
\begin{equation} \label{eq:controlemoicamec}
    |\langle g, \varphi \rangle_{\Sp'(\R)\times\Sp(\R)} |
    = 
    \frac{|\langle g, \varphi_n \rangle_{\Sp'(\R)\times\Sp(\R)} | }{n+1} 
    = 
    \frac{|\langle g, \varphi_n \rangle_{\MR\times\CR} | }{n+1}
    \leq \frac{\lVert g \rVert_{\MR} \lVert \varphi \rVert_\infty}{n+1}.
\end{equation}
This implies that $| \langle g, \varphi \rangle_{\Sp'(\R)\times\Sp(\R)} | \leq  \lim_{n\rightarrow \infty}  \frac{\lVert g \rVert_{\MR} \lVert \varphi \rVert_\infty}{n+1} = 0$ for any $\varphi \in \Sp(\R)$, hence $g = 0$. The function $f$, which satisfies $\LT f = g = 0$, is therefore a polynomial of degree at most $(N_T-1)$. Since $f$ is also periodic, we deduce that $f$ is constant. Finally, $f \in \MLSz$ has mean 0, hence $f = 0$, as expected.
\end{proof}


Proposition~\ref{prop:directsumdecomposition} allows us to identify the native space 
of the optimization problem~\eqref{eq:firstoptiproblem} as follows.
 
\begin{definition} \label{def:STnative}
    Let $\LT$ and $\LS$ be trend- and seasonal-admissible operators, respectively. We define the \textit{seasonal-trend native space} associated to $(\LT,\LS)$ as
\begin{equation} \label{eq:STnative}
    \MLTLS = \MLT \times \MLS \text{ if } N_T=0, \quad \text{and } \MLTLS = \MLT \times \MLSz \text{ if } N_T\geq 1. 
\end{equation}
\end{definition}
 
The distinction in \eqref{eq:STnative} reflects the fact that the constant function, which is always in $\MLS$, is in $\MLS$ if and if $\LT \{1\} = 0$. In this case, the decomposition $f = f_S + f_T \in \MLT + \MLS$ is not unique since 
any decomposition $f = (f_S + \alpha) + (f_T - \alpha)$ would be valid for any $\alpha \in \R$. We impose that $f_S$ has $0$-mean to recover the uniqueness.
The native space $\MLTLS$ is then such that 
\begin{itemize}
\item the regularization $\mathcal{R}(f_T,f_S) = \| f_T \|_{\MR} + \| f_S \|_{\MS}$ is finite for $(f_T,f_S) \in \MLTLS$; and 
\item the components $(f_S, f_T)$ can be uniquely recovered from the sum $f = f_T + f_S \in \MLTLS$, with (i) $f_T \in \MLT$ and (ii) $f_S \in \MLS$ or $f_S \in \MLSz$ depending if $N_T=0$ or $N_T\geq 1$. 
\end{itemize}

The native space $\MLTLS$ inherits the Banach space structure of its component spaces $\MLT$ and $\MLS$ (or $\MLSz$).

\begin{proposition} \label{prop:banachnative}
Let $\LT$ and $\LS$ be trend-admissible and seasonal-admissible operators and $1 \leq p \leq \infty$.
The seasonal-trend native space $\MLTLS$ is a Banach space for the norm 
\begin{equation} \label{eq:normMLTLSpfinite}
    \| (f_T , f_S) \|_{\MLTLS,p} = \left\| ( \| f _T  \|_{\MLT} , \| f _S  \|_{\MLS} ) \right\|_p =  \left( \| f _T  \|_{\MLT}^p + \| f_S \|_{\MLS}^p \right)^{1/p} 
\end{equation}
if $p < \infty$ and
\begin{equation} \label{eq:normMLTLSpinfinite}
    \| (f_T , f_S) \|_{\MLTLS,\infty} = \left\| ( \| f _T  \|_{\MLT} , \| f _S  \|_{\MLS} ) \right\|_\infty = \max \left( \| f _T  \|_{\MLT} , \| f_S \|_{\MLS} \right) 
\end{equation}
if $p= \infty$. 
\end{proposition}

\begin{proof}
If $\mathcal{X}$ and $\mathcal{Y}$ are two Banach spaces with respective norms $\|\cdot \|_{\mathcal{X}}$ and $\|\cdot \|_{\mathcal{Y}}$, then the product space $\mathcal{Z} = \mathcal{X}\times \mathcal{Y}$ is easily shown to be a Banach space for the norm $\| z \|_{\mathcal{Z},p} = \| (x,y) \|_{\mathcal{X}\times \mathcal{Y}, p} =  \left\| \left( \|x \|_{\mathcal{X}},  \|y\|_{\mathcal{Y}} \right) \right\|_p$
~\cite[Section 3]{unser2022convex}. 
The result then follows by applying this result to $(\mathcal{X}, \|\cdot \|_{\mathcal{X}}) = (\MLT, \| \cdot \|_{\MLT})$ and  $(\mathcal{Y}, \|\cdot \|_{\mathcal{Y}}) = (\MLS, \| \cdot \|_{\MLS})$ or $(\MLSz, \| \cdot \|_{\MLS})$.
\end{proof}
    
\section{Sensing Functionals}\label{sec:sensing}

In this section, we identify the class of linear functionals $\phi_\ell$ that can be used in the linear sensing process $\bm{\Phi}(f_T+f_S) = ( \phi_{\ell} (f_T+f_S) )_{1\leq \ell \leq L}$ of the optimization problem~\eqref{eq:STnative}. 

    \subsection{The Preduals of the Trend and Seasonal Native Spaces} \label{sec:preduality1}

    One our main goal in this article is to reveal the structure of the solutions of \eqref{eq:firstoptiproblem}. 
    This will be done in Section~\ref{sec:theRT} in the form of a representer theorem (see Theorem~\ref{theo:RTfull}). 
    The existing literature on representer theorems over Banach spaces revealed that  the \textit{predual} of the native Banach space plays a dominant role~\cite{Fisher1975,unser2017splines,fageot2020tv}. 
    The predual of $\MLTLS$ will be identified in Section~\ref{sec:preduality2} and is
    built on the spaces  $\CLT$ and $\CLS$, which are preduals of $\MLT$ and $\MLS$, respectively. 
    
    \paragraph{The space $\CLT$.}
    We fix a trend-admissible operator $\LT$. The space $\CLT$ is the predual of $\MLT$. It is called the \textit{trend measurement space} since it characterizes the possible function that can be used in the measurement process of the trend component. As we recalled on Section~\ref{sec:radon}, the space of continuous and vanishing functions $\CR$ is the predual of $\MR$. Hence, the spaces $\CLT$ generalize the space $\CR$. The adjoint $\LT^*$ of $\LT$ is the unique trend admissible operator such that 
    \begin{equation}
        \langle \LT \varphi_1 , \varphi_2 \rangle_{\Sp'(\R)\times\Sp(\R)} = \langle  \LT^* \varphi_2 ,  \varphi_1 \rangle_{\Sp'(\R)\times\Sp(\R)}
    \end{equation}
    for any test functions $\varphi_1, \varphi_2 \in \mathcal{S}(\R)$. For instance, $\Dop^* = -\Dop$. 
    
    \begin{itemize}
    
        \item  If $\LT$ is \textbf{invertible}, then $\CLT = \LT^{*} \{ \CR \}$  inherits the Banach space structure of $\CR$ for the norm 
        \begin{equation} \label{eq:posetanorme}
            f \mapsto \| f \|_{\CLT} = \| \LT^{-1*} f \|_{\infty}. 
        \end{equation}
        
        \item The trend measurement space of a non-invertible operator is harder to characterize and its precise structure will note be used thereafter. 
        We therefore refer the interested reader to the works~\cite{unser2017splines,unser2019native} where they are constructed and studied. 
        It suffices to know that $\CLT$ is a Banach space for a norm $\| \cdot \|_{\CLT}$ that generalizes~\eqref{eq:posetanorme}. 
        The key is to construct a dual space of the null space $\mathcal{P}_{N_T}(\R)$ of $\LT$ (see Section~\ref{sec:trendop}) such that $\CLT = \LT^* \CR \oplus (\mathcal{P}_{N_T}(\R))'$ and to specify the Banach norm using this direct-sum decomposition. A similar and simpler construction is given for the seasonal space $\CLS$ thereafter.  
        
    \end{itemize}
\paragraph{The spaces $\CLS$ and $\CLSz$.} 
We fix a seasonal-admissible operator $\LS$. 
As for the trend case, the space $\CLS$, called the \textit{seasonal measurement space}, is the predual of the space $\MLS$.
The periodic adjoint $\LS^*$ of $\LS$ is the unique periodic operator such that 
\begin{equation}
    \langle \LT \varphi_1 , \varphi_2 \rangle_{\Sp'(\T)\times\Sp(\T)} = \langle \LT^* \varphi_2 ,  \varphi_1  \rangle_{\Sp'(\T)\times\Sp(\T)}
\end{equation}
for any periodic test functions $\varphi_1, \varphi_2 \in \mathcal{S}(\T)$. The construction and analysis of seasonal measurement spaces is detailed in~\cite[Section 3.3]{fageot2020tv}, from which we recall the following facts. 
  
\begin{itemize}
    \item If $\LS$ is {invertible}, then $\CLS = \LS^{*} \{ \CT \}$ inherits the Banach space structure of $\CT$ for the norm 
    \begin{equation} \label{eq:posetanorme}
        f \mapsto \| f \|_{\CLS} = \| \LS^{-1*} f \|_{\infty}. 
    \end{equation}
    
    \item If $\LS$ is not invertible, its null space is made of constant functions. In this case, we have that $\CLS = \LS^* \CT \oplus \mathrm{Span}\{1\}$. The norm of $g \in \CLS$ is then given by
    \begin{equation}
        \| g \|_{\CLS} = \max \left(  \| h \|_{\infty} , |\alpha| \right) 
    \end{equation}
    where $h \in \CT$ is such that $\LS^* h = g$ and $\widehat{h}[0] = 0$ and where $\alpha = \widehat{g}[0] = \int_\T g(t) \mathrm{d} t \in \R$. Moreover, $(\CLS, \|\cdot \|_{\CLS} )$ is a Banach space~\cite[Theorem 2]{fageot2020tv}. 
\end{itemize} 

We moreover define the space of $0$-mean functions in $\CLS$ as
\begin{equation}
    \CLSz = \{ g \in \CLS, \quad \widehat{g}[0] = 0\}. 
\end{equation}
It is a Banach space for the norm $\|\cdot \|_{\CLT}$. 

\subsection{The Predual of the Native Space} \label{sec:preduality2}

We define the new function spaces that will be shown to be the preduals of the seasonal-trend native spaces introduced in Section~\ref{sec:STspace} in Proposition~\ref{prop:generalizedRieszMarkov}. 

    
\begin{definition}
    Let $\LT$ and $\LS$ be respectively trend-admissible and seasonal-admissible operators, with respective orders $N_T$ and $N_S$.     
    We define
\begin{equation} \label{eq:STmeasurement}
    \CLTLS = \CLT \times \CLS \text{ if } N_T=0, \quad \text{and } \CLTLS = \CLT \times \CLSz \text{ if } N_T\geq 1. 
\end{equation}
We call $\CLTLS$ the \emph{seasonal-trend measurement space} associated to the pair $(\LT,\LS)$. 
\end{definition}

We  fix $1 \leq q \leq \infty$. Then, we endow $\CLTLS$ with the norm
\begin{equation}
    \| (g_T , g_S) \|_{\CLTLS,q} = \left( \| g _T  \|_{\CLT}^q +  \| g_S \|_{\CLS}^q \right)^{1/q}
\end{equation}
if $q < \infty$, and
\begin{equation}
    \| (g_T , g_S) \|_{\CLTLS,\infty} = \max \left( \| g _T  \|_{\CLT} , \| g_S \|_{\CLS} \right)
\end{equation}
if $q = \infty$.

\begin{proposition}
\label{prop:generalizedRieszMarkov}
Let $1 \leq p , q \leq \infty$ be such that $\frac{1}{p}+\frac{1}{q} = 1$. The following facts hold.
\begin{itemize}
    \item     The space $(\CLTLS, \| \cdot \|_{\CLTLS,q} ) $ is a Banach space. 

    \item The  topological dual of $\CLTLS$ is the native space $(\CLTLS )' = \MLTLS$. 
    
    \item The dual norm associated to the norm $\| \cdot \|_{\CLTLS,q}$ is the norm $\| \cdot \|_{\MLTLS,p}$ defined in \eqref{eq:normMLTLSpfinite} and \eqref{eq:normMLTLSpinfinite}. This means that
     \begin{align}
    \| (f_T , f_S) \|_{\MLTLS,p} 
    &=  \sup_{\substack{(g_T, g_S) \in \CLTLS \\ \| (g_T , g_S) \|_{\CLTLS,q} = 1}}
    \langle (f_T, f_S) , (g_T, g_S) \rangle_{\MLTLS\times \CLTLS}  \nonumber \\
    &= 
    \sup_{\substack{(g_T, g_S) \in \CLTLS \\ \| (g_T , g_S) \|_{\CLTLS,q} = 1}} \langle f_T, g_T \rangle_{\MLT\times \CLT} + \langle f_S, g_S \rangle_{\MLS\times \CLS}.
\end{align}
\end{itemize}
\end{proposition}

\begin{proof}
We use~\cite[Lemma 1]{unser2022convex}, whose first statement directly implies the following result.
Let $\mathcal{X}$ and $\mathcal{Y}$ be two Banach spaces with respective norms $\|\cdot \|_{\mathcal{X}}$ and $\|\cdot \|_{\mathcal{Y}}$. Their topological dual and their dual norms are denoted by $(\mathcal{X}',\| \cdot \|_{\mathcal{X}'})$ and $(\mathcal{Y}',\| \cdot \|_{\mathcal{Y}'})$.
Then, the product space $\mathcal{Z} = \mathcal{X}\times \mathcal{Y}$ is a Banach space for the norm 
\begin{equation} \label{eq:Zp}
    \| z \|_{\mathcal{Z}} = \| (x,y) \|_{\mathcal{X}\times \mathcal{Y},p} =  \left\| \left( \|x \|_{\mathcal{X}} , \|y\|_{\mathcal{Y}} \right) \right\|_p.
\end{equation}
Its topological dual is the space $\mathcal{Z}' = \mathcal{X}' \times \mathcal{Y}'$, whose dual norm is given by 
\begin{equation} \label{eq:Zprimeq}
    \| z' \|_{\mathcal{Z}'} = \| (x',y') \|_{\mathcal{X}'\times \mathcal{Y}',q} = \left\| \left( \|x' \|_{\mathcal{X}'} , \|y'\|_{\mathcal{Y}'} \right) \right\|_q.
\end{equation}
We apply this result to the spaces $\mathcal{X} = \CLT$ and $\mathcal{Y} = \CLS$ if $N_T=0$ or $\CLSz$ if $N_T \geq 1$. We know moreover from existing works that $\CLT' = \MLT$~\cite[Theorem 6]{unser2017splines}, $\CLS' = \MLS$, and $\CLSz' = \MLSz$~\cite[Theorem 3]{fageot2020tv}. This implies the desired result. 
\end{proof}

Proposition~\ref{prop:generalizedRieszMarkov} shows that $(\CLTLS,\MLTLS)$ is a dual pair. 
It generalizes known results for the pairs $(\CLT,\MLT)$ and $(\CLS,\MLS)$. 
They are themselves generalizations of the Riesz-Markov theorem, which states the duality of the pairs  $(\CR,\MR)$ and $(\CT,\MS)$~\cite{gray1984shaping}.  Proposition~\ref{prop:generalizedRieszMarkov} is therefore a seasonal-trend generalized Riesz Markov theorem.

\subsection{Periodization and Seasonal-Trend Sensing Functionals} \label{sec:thetruesensing}

In the previous section, we identified the predual $\CLTLS$ of the seasonal-trend native space $\MLTLS$. In particular, any $(g_T,g_S) \in \CLTLS $ acts linearly and weak*-continuously on the native space $\MLTLS$ via
\begin{equation}
    (g_T,g_S) \quad : \quad (f_T,f_S) \mapsto \langle f_T, g_T \rangle_{\MLT\times \CLT} + \langle f_S, g_S \rangle_{\MLS\times \CLS}.
\end{equation}
This is not directly applicable to the optimization problem \eqref{eq:firstoptiproblem}, where the action is on the sum $f_T+f_S$. 
We shall therefore identify a class of sensing functionals $\phi \in \Sp'(\R)$ such that $(\phi,\Per\{\phi\}) \in \CLT \times \CLS$ can be used as an element of $\CLTLS$.

\begin{definition}
Let $f \in \Sp'(\R)$. For $N \geq 0$, we set $f_N = \sum_{|n|\leq N} f(\cdot - n)$. If the sequence $(f_N)_{N\geq 0}$ converges in $\Sp'(\R)$, then its limit, denoted by
\begin{equation}
    \Per\{f\} = \sum_{n\in \Z} f(\cdot - n),
\end{equation}
exists in $\Sp'(\R)$. In this case, $\Per\{f\}$ is in $\Sp'(\T)$ and we call it the \emph{periodization of $f$}. A generalized function is \emph{periodizable} if its periodization exists.      
\end{definition}

Any test function $\varphi$ in $\Sp(\R)$ is periodizable, and the sum $\varphi_N = \sum_{|n|\leq N} \varphi(\cdot - n)$ even converges in $\Sp(\T)$.
We provide sufficient conditions to ensure that a given generalized function is periodizable\footnote{Some generalized functions are not periodizable. For instance, the constant $f=1$ is such that $f_N = 2N+1$, hence $(f_N)_{N\geq 1}$ does not converge in $\Sp'(\R)$.}.

\begin{proposition} \label{prop:periodizable}
The following statements hold.
\begin{itemize}
    \item If $f \in \mathcal{O}_C'(\R)$, then $f$ is periodizable with $\Per \{ f \} \in \Sp'(\T)$. Moreover, any $g \in \Sp'(\T)$ can be written as the periodization $g = \Per \{f\}$ of some $f \in \mathcal{O}_C'(\R)$. 
    \item Assume that $f \in L_1(\R)$. Then, $f$ is periodizable and $\Per \{f\} \in L_1(\T)$.
    \item If $f : \R \rightarrow \R$ is continuous and such that
\begin{equation} \label{eq:conditionfforperiod}
    \left| f(t) \right| \leq \frac{C}{1+|t|^p}
\end{equation}
for some $C> 0$, $p>1$, and any $t \in \R$, then $f$ is periodizable and $\Per\{f\} \in \mathcal{C}(\T)$.
\end{itemize}
\end{proposition}

\begin{proof}
The first part of the first point is a classic property of the (not so classic) space $\mathcal{O}'_C(\R)$ and is shown in~\cite[Section 4.2]{fischer2015duality}.
The second part follows from the existence of smooth partitions of unity: there exists functions $\varphi \geq 0$ in $\mathcal{S}(\R)$ such that $1 = \sum_{n \in \Z} \varphi(\cdot - n)$. Said differently, $\varphi$ is periodizable and $\Per \{\varphi\} = 1$. 

We fix $g \in \Sp'(\T)$. We observe that $f_N = \sum_{|n|\leq N} f(\cdot - n) = g .\sum_{|n|\leq N} \varphi (\cdot - n) = g. \varphi_N$, which converges towards $g . \Per \{\varphi\} = g$ in $\Sp'(\T)$. Hence, $g = \Per \{f\}$ is the periodization of some function $f \in \mathcal{O}'_C(\R)$. \\

Let $f \in L_1 (\R)$. Then, $(f_N)_{N\geq 1}$ is a Cauchy sequence in $L_1([0,1])$. Indeed, we have that, for any $1 \leq M \leq N$, 
\begin{align}
\int_0^1 |f_N (t) - f_M(t)|\mathrm{d}t
&= 
\int_0^1 \left| \sum_{M+1 \leq |n| \leq N} f(t - n) \right| \mathrm{d} t
\leq 
\int_{0}^1 \left( \sum_{M+1 \leq |n| \leq N}   \left|  f(t - n) \right| \right) \mathrm{d} t \nonumber \\
&\leq
\int_{M \leq |t| \leq {N+1}} |f(t)|\mathrm{d}t
\end{align}
which can be made arbitrarily small for large $N,M$ as $f$ is integrable over $\R$.
The space $L_1([0,1])$ is complete, hence the sequence $(f_N)_{N\geq 1}$ converges to some function $f_\infty = \sum_{n\in\Z} f(t-n) \in L_1([0,1])$, which can be seen as a periodic function over the real line $\R$. \\

For the last part of Proposition~\ref{prop:periodizable}, we observe that 
$|f_N(t)|  \leq \sum_{|n|\leq N} |f(t-n)| \leq \sum_{|n|\leq N} \frac{C}{1+|t-n|^p}$, which is absolutely convergent. Hence, the limit $\Per\{f\} (t) = \sum_{n\in\Z}  f(t-n) $ exists for any $t \in \R$ and is easily shown to be $1$-periodic. Then, 
\begin{equation}
    |f_N (t) - \Per\{f\}(t) | \leq  \sum_{|n - 1| \geq  N} |f(t-n)| \leq
     \sum_{|n - 1| \geq  N}  \frac{C}{1+|t-n|^p},
\end{equation}
the latter converging  uniformly to $0$  over compact intervals. Hence, $f_N$ converges to $\Per \{f\}$ uniformly on compact intervals. This implies in particular that $\Per\{f\}$ is continuous as the uniform limit of continuous functions on any compact interval. Hence, $f$ is periodizable with continuous periodization.
\end{proof}

Proposition~\ref{prop:periodizable} is useful due to the importance of the periodization to interpret $\phi(f_T+f_S)$ for $(f_T,f_S) \in \MLTLS$. Other conditions than \eqref{eq:conditionfforperiod} lead to identical conclusions. As seen in the proof, it suffices that $ \sum_{|n - 1| \geq  N} |f(t-n)| $ converges uniformly to $0$ on compact intervals to imply that $\Per\{f\}$ exists and is continuous.

\begin{definition} \label{def:sensingadmissible}
Let $\LT$ and $\LS$ be trend-admissible and seasonal-admissible operators, with respective orders $N_T, N_S \geq 0$.
A generalized function $\phi \in \Sp'(\R)$ is \emph{seasonal-trend sensing admissible} for the pair $(\LT,\LS)$ (or $(\LT,\LS)$-sensing admissible) if $\phi$ is periodizable and $(\phi, \Per\{\phi\}) \in \CLT \times \CLS$.
Then, $\phi$ acts on the trend-seasonable native space $\MLTLS$ as
\begin{equation}
    \phi(f_T, f_S) = \langle f_T , \phi \rangle_{\MLT,\CLT} +  \langle f_S , \Per\{\phi\} \rangle_{\MLS,\CLS}
\end{equation}
for any $(f_T,f_S) \in \MLTLS$. 
\end{definition}

Due to the unique decomposition $f = f_T + f_S \in \MLTLS$, we use the slight abuse of notation $ \phi(f_T, f_S) = \phi(f_T + f_S)$. 
For smooth functions $f_T \in \Sp(\R)$, $f_S \in \Sp(\T)$ and test function $\phi \in \Sp(\R)$ (which is periodizable), we have by using the periodicity of $f_S$ that
\begin{align}
    \phi(f_T + f_S) &= \int_{\R} \phi(t) (f_T + f_S)(t) \, \mathrm{d} t  =  \int_{\R} \phi(t) f_T(t) \, \mathrm{d} t  + \int_{\R} \phi(t) f_S(t) \, \mathrm{d} t 
    \nonumber \\
    &=  \int_{\R} \phi(t) f_T(t) \, \mathrm{d} t  + \sum_{n \in \Z} \int_{n}^{n+1} \phi(t) f_S(t) \, \mathrm{d} t  
    =  \int_{\R} \phi(t) f_T(t) \, \mathrm{d} t  +  \int_{0}^1  \left( \sum_{n \in \Z}\phi(t-n)\right)  f_S(t) \, \mathrm{d} t 
    \nonumber \\ 
    &=  \int_{\R} \phi(t) f_T(t) \, \mathrm{d} t  + \int_{0}^1\Per\{\phi\}(t) f_S(t) \, \mathrm{d} t 
    \nonumber \\ 
    &= \langle f_T , \phi \rangle_{\MLT,\CLT} +  \langle f_S , \Per\{\phi\} \rangle_{\MLS,\CLS}. 
\end{align}
The goal of Definition~\ref{def:sensingadmissible} is then to generalize this principle to the largest possible space for which it is valid. 
We exemplify Definition~\ref{def:sensingadmissible} on different pairs of operators and different measurement functionals. 

\begin{itemize}
    \item  Set $(\LT, \LS) = (\mathrm{Id}, \mathrm{Id})$. Then, $\phi$ is seasonal-trend sensing admissible if and only if $\phi \in \mathcal{C}_0(\R)$, $\phi$ is periodizable, and $\Per\{\phi\} \in \mathcal{C}(\T)$. 
    
    It is worth noting that there exist functions $\phi \in \mathcal{C}_0(\R)$ that are not-periodizable. The continuous and vanishing function $\phi (t) = \frac{1}{1+|t|}$ is such that $\phi_N(t) = \sum_{|n|\leq N} \frac{1}{1+|t-n|} \rightarrow \infty$ when $N \rightarrow \infty$ at any time $t \in \R$ and is therefore not periodizable.
    
    There also  exist periodizable functions $\phi \in \mathcal{C}_0(\R)$ such that $\Per\{\varphi\}$ is not continuous. Consider the continuous function $\phi$, which connects piecewise-linearly the constraints $\phi(n) = \frac{1}{n}$, $\phi(n-\frac{1}{n}) = \phi(n+\frac{1}{n}) = 0$. Then, $\phi$
    is in $L_1(\R)$ --- since $\int_{n-1/n}^{n+1/n} \phi(t) \mathrm{d}t = \frac{1}{n^2}$, hence $\int_{\R} |\phi(t)|\mathrm{d}t = \sum_{n\geq 2} \frac{1}{n^2} < \infty$ --- and therefore periodizable in $L_1(\T)$ according to Proposition~\ref{prop:periodizable}.
    However, $\phi_N(0) = \sum_{2 \leq n \leq N} \frac{1}{n}$ diverges when $N\rightarrow \infty$, and $\Per\{\phi\}$ is unbounded and therefore discontinuous.
  
    \item Set $\LT = \LS = \Dop^2$ and $\phi = \delta$ such that $\Per\{\phi\} = \mathrm{III}$. It is known that $\delta \in \mathcal{C}_{\Dop^2}(\R)$~\cite[Theorem 1]{unser2019representer} and $\mathrm{III} \in \mathcal{C}_{\Dop^2}(\T)$~\cite[Proposition 9]{fageot2020tv}. Hence, we are in the conditions of Definition~\ref{def:sensingadmissible} and we have
    \begin{equation}
        \phi(f_T + f_S) = (f_T + f_S)(0) = \langle f_T , \delta \rangle_{\mathcal{M}_{\Dop^2}(\R),\mathcal{C}_{\Dop^2}(\R)} +  \langle f_S , \mathrm{III} \rangle_{\mathcal{M}_{\Dop^2}(\T),\mathcal{C}_{\Dop^2}(\T)}.
    \end{equation}
    The case of spatial sampling will be treated more extensively in Section~\ref{sec:sampling}. 
    
    \item Let $\phi \in \Sp(\R)$ be a partition of unity; \emph{i.e.}, $\phi \geq 0$ satisfies $\Per \{\phi\} = \sum_{n\in \Z} \phi(\cdot - n) = 1$. Then, for any pair $(\LT,\LS)$, we have that $\varphi \in \CLT$ and $\Per \{\varphi\} = 1 \in \CLS$. Hence, we are in the conditions of Definition~\ref{def:sensingadmissible} and we have
    \begin{equation}
        \phi(f_T + f_S) = \langle f_T , \phi \rangle_{\MLT,\CLT} +  \langle f_S ,  1 \rangle_{\MLS, \CLS} = \langle f_T , \phi \rangle_{\Sp'(\R),\Sp(\R)} + \widehat{f}_S[0],
    \end{equation}
    where $\widehat{f}_S[0] = \langle f_S , 1 \rangle_{\Sp'(\T),\Sp(\T)} $ is the mean of the periodic generalized function $f_S \in \Sp'(\Z)$. 
    
    \item Set $\LT = (\Dop + \mathrm{Id})$ and $\phi= \One_{[0,\tau)}$ for some integer $\tau\geq1$.
    We first observe that 
    $\Per\{\phi\} = \sum_{k \in \Z} \One_{[k,k+\tau)} = \tau$ is constant, hence in $\CLS$. Moreover, $\phi \in \mathcal{C}_{\Dop + \mathrm{Id}}(\R)$~\cite[Theorem 6]{unser2017splines} and $\langle f_T , \phi \rangle_{\mathcal{M}_{\Dop}(\R),\mathcal{C}_{\Dop}(\R)} = \int_{0}^\tau f_T(t) \mathrm{d} t$. Hence, $\phi$ is $(\Dop + \mathrm{Id},\LS)$-sensing admissible and we have
    \begin{equation}
        \phi(f_T + f_S) = \int_{0}^\tau f_T(t) \mathrm{d} t + \tau \widehat{f}_S[0].
    \end{equation}       
    
    \item Let $\LT = \Dop$ and $\phi = \One_{[0,1)}$. Then, as for the previous case, we easily see that $\phi$ is $(\Dop,\LS)$-sensing admissible. However, any $f_S$ such that $(f_T, f_S) \in \mathcal{M}(\Dop,\LS)$ has $0$ mean. Hence, we have that 
    \begin{equation}
        \phi(f_T + f_S) = \int_{0}^1 f_T(t) \mathrm{d} t.
    \end{equation}    
\end{itemize}

\subsection{Sampling Admissibility} \label{sec:sampling}

In view of the Representer Theorem (see Section~\ref{sec:theRT}), spatial sampling is an admissible sensing process for the seasonal-trend optimization problem~\eqref{eq:firstoptiproblem} if $\phi_\ell : (f_T,f_S) \mapsto f_T(t_\ell) + f_S(t_\ell)$ is sensing admissible for any position $t_\ell \in \R$. This leads to the formal definition. 

\begin{definition}
A pair $(\LT,\LS)$ of trend- and seasonal-admissible operators is said to be \emph{sampling admissible} if $\delta(\cdot- t_0)$ is $(\LT,\LS)$-sensing admissible for any $t_0 \in \R$. 
\end{definition}

The Dirac impulse $\delta$ is periodizable with $\Per\{\delta\} = \mathrm{III}$, hence the sampling admissibility has no periodizability issue. 
In Proposition~\ref{prop:examplessamplingadm}, we characterize sampling admissibility for invertible and derivative operators.

\begin{proposition} \label{prop:examplessamplingadm}
Let $\LT$ and $\LS$ be trend- and seasonal-admissible operators, respectively. Then, the following fact holds.
\begin{itemize}
    \item The pair $(\LT,\LS)$ is sampling-admissible if and only if 
    \begin{equation} \label{eq:equifirst}
        \delta \in \CLT \quad \text{and} \quad  \mathrm{III} \in \CLS. 
    \end{equation}
    \item If $\LT$ and $\LS$ are invertible with respective Green's and periodic Green's function $\psi_{\LT} = \LT^{-1}\delta $ and $\rho_{\LS} = \LS^{-1} \mathrm{III}$, then $(\LT,\LS)$ is sampling-admissible if and only if
    \begin{equation} \label{eq:newequi}
         \psi_{\LT} \in \mathcal{C}_0(\R) \quad \text{and} \quad \rho_{\LS} \in \mathcal{C}(\T). 
    \end{equation}
    \item If $\LT = (\mathrm{Id} - \Delta)^{\gamma_T / 2}$ and $\LS= (\mathrm{Id} - \Delta)^{\gamma_S / 2}$ are Sobolev operators with parameters $\gamma_T, \gamma_S \geq 0$, then $(\LT,\LS)$ is sensing admissible if and only if $\gamma_T > 1$ and $\gamma_S > 1$.
    \item Let $N_T,N_S \geq 1$. Then, $(\Dop^{N_T}, \Dop^{N_S})$ is sampling-admissible if and only if $N_T \geq 2$ and $N_S \geq 2$.
\end{itemize} 
\end{proposition}

\begin{proof}
First of all, all the considered function spaces are shift-invariant, hence $\delta(\cdot - t_0)$ is sensing admissible if and only if $\delta$ is. We can therefore restrict to $t_0 = 0$. The Dirac impulse $\delta$ is periodizable with $\Per\{\delta\} = \mathrm{III}$, hence $(\LT,\LS)$ is sampling admissible if and only if $\delta \in \CLT$ and $\mathrm{III} \in \CLS$. \\

Let us assume that both $\LT$ and $\LS$ are invertible. 
Then, by the definition of the trend-measurement space, $\delta \in \CLT = \LT^{*} \mathcal{C}_0(\R)$ if and only if $\psi_{\LT} = \LT^{-1} \delta \in \mathcal{C}_0(\R)$. Similarly, $\mathrm{III} \in \CLS = \LS^{*} \mathcal{C}(\T)$ if and only if $\rho_{\LS} = \LS^{-1} \mathrm{III} \in \mathcal{C}(\T)$. This implies the equivalence between \eqref{eq:equifirst} and \eqref{eq:newequi}. \\

The operators $\LT = (\mathrm{Id} - \Delta)^{\gamma_T / 2}$ and $\LS = (\mathrm{Id} - \Delta)^{\gamma_S / 2}$ are invertible, hence the criterion \eqref{eq:newequi} applies. Then, the Green's function of $\LS$ is continuous if and only if $\gamma_S > 1$ according to \cite[Proposition 9]{fageot2020tv}.      
For the trend, we observe that $\psi_{\LT}$ is the inverse Fourier transform of $\omega \mapsto F_{\gamma_T}(\omega) = (1 + \omega^2)^{-\gamma_T/2}$. 

We first consider $0 < \gamma_T < 1$. Then, for any $\omega \in \mathbb{R}$,
\begin{equation}
    F_{\gamma_T}(\omega) = |\omega|^{-\gamma_T} + \left( (1 + \omega^2)^{-\gamma_T/2} - |\omega|^{-\gamma_T} \right) = |\omega|^{-\gamma_T} + G_{\gamma_T} (\omega)
\end{equation}
with $G_{\gamma_T} \in L_1(\R)$. Indeed, $G_{\gamma_T}$ behaves like $- |\omega|^{-\gamma_T}$ around the origin and like $-\frac{\gamma_T}{2} |\omega|^{-(\gamma_T + 2)}$ at infinity, both being integrable. The inverse Fourier transform of $G_{\gamma_T}$ is therefore continuous. The inverse Fourier transform of $|\omega|^{-\gamma_T}$ is known to be proportional to $x \mapsto |x|^{1 - \gamma_T}$~\cite[Table 1 p. 359]{gelfand1964generalized} and is therefore discontinuous (and unbounded) around the origin. Hence, $\psi_{\LT}$ is discontinuous as the sum of continuous and discontinuous functions.

For $\gamma_T = 1$, we use that the inverse Fourier transform of $F_{\gamma_T}$ is proportional to $K_0(|x|)$, where $K_0$ is the modified Bessel function of the second kind. The latter is known to diverge like $\log |x|$ around the origin, hence $\psi_{\LT}$ is also not continuous.

We have proved that, for $\LT = (\mathrm{Id} - \Delta)^{\gamma_T / 2}$ and $\LS = (\mathrm{Id} - \Delta)^{\gamma_S / 2}$, $\psi_{\LT}$ is continuous if and only if $\gamma_T > 1$ and $\rho_{\LT}$ is continuous if and only if $\gamma_S > 1$. This proves the third point of Proposition~\ref{prop:examplessamplingadm}. \\

For the last point, we rely on \eqref{eq:equifirst}. The condition $\mathrm{III} \in \mathcal{C}_{\mathrm{D}^{N_S}}(\T)$ is equivalent to $N_S \geq 2$ according to~\cite[Proposition 9]{fageot2020tv}.
The Dirac impulse $\delta$ is shown to be in $\mathcal{C}_{\Dop^2}(\R)$ in~\cite[Proposition 15]{unser2019representer}, and the reasoning can be extended to higher-order derivatives with $N_T \geq 2$.
On the contrary, $\delta$ is not in $\mathcal{C}_{\Dop}(\R)$. Otherwise, we would have that $\delta = \Dop g + \varphi$ with $g \in \mathcal{C}_0(\R)$ and $\varphi \in \Sp(\R)$~\cite{unser2017splines}. Since $\delta = \Dop u$ with $u(t) = \One_{[0,\infty)}(t)$ is the Heaviside function, this would imply that $\Dop( u - g ) = \varphi$ is infinitely smooth, hence $u - g$ is infinitely smooth which is not the case. This concludes the proof. 
\end{proof}

\section{Seasonal-Trend Representer Theorem}\label{sec:RT}

We aim at recovering a composite function $f_0 : \R \rightarrow \R$ given by $f_0 = f_{0,S} + f_{0,T}$ where $f_{0,S}$ is a $1$-periodic function --- the seasonal part of $f_0$ --- and $f_{0,T}$ is a non-periodic function --- the trend part. 
We have access to $L \geq 1$ noisy observations 
\begin{equation}
  \bm{y} \approx \bm{\Phi}(f_0) \in \R^L,
\end{equation}
where $\bm{\Phi}(f_0)$ depends linearly on $f_0$. 
The reconstruction model depends on the choice of trend- and seasonal-admissible operators $\LT$ and $\LS$, which determines the seasonal-trend native space $\MLTLS$ (see Definition \eqref{def:STnative}) for the reconstruction task. 

In this section, we obtain the representer theorem for the reconstruction problem~\eqref{eq:firstoptiproblem}.
We first obtain a seasonal-trend representer theorem for measures in Section~\ref{sec:RTmeasure} and then deduce the general case in Section~\ref{sec:theRT}. We deal with the special case of piecewise-polynomial reconstruction in Section~\ref{sec:derivatives} and we compare with quadratic regularizations for seasonal-trend reconstruction in Section~\ref{sec:hilbert}. 

\subsection{Seasonal-Trend Representer Theorem for Radon Measures} \label{sec:RTmeasure}

We first consider optimization problems such as~\eqref{eq:firstoptiproblem} with no regularization operators, \emph{i.e.} with $\LS = \LT = \mathrm{Id}$. According to Definition~\ref{def:STnative}, the seasonal-trend native space is  $\mathcal{M}(\mathrm{Id}, \mathrm{Id}) = \MR \times \MS$.
In preparation for the general representer theorem, we slightly modify the optimization problem by adding a finite-dimensional vector in the search space, which only affects the measurement process and is not concerned by the regularization.
Proposition~\ref{prop:RTmeasure} will be the cornerstone for the proof of the general case (see Theorem~\ref{theo:RTfull} in Section~\ref{sec:theRT}).

\begin{proposition}[Seasonal-Trend Representer Theorem for Radon Measures] \label{prop:RTmeasure}
We consider the following:
\begin{itemize}
\item $N\geq 0$ and $L\geq 1$ are some integers such that $N\leq L$;
\item $\mathcal{E} : \R^L \rightarrow \R^+ \cup \{\infty\}$ is a proper, lower semi-continuous, coercive, and convex data fidelity functional,
\item $\bm{\Phi} = (\phi_\ell)_{1 \leq \ell \leq L} = \left( (g_{T, \ell}, g_{S,\ell}, \bm{z}_\ell) \right)_{1 \leq \ell \leq L} \in \left( \mathcal{C}_0(\R) \times \mathcal{C}(\T) \times \R^N \right)^L$ is a family such that (i) $\bm{\Phi} ( 0, 0 , \bm{c} ) = \bm{0}$ implies that $\bm{c} = \bm{0}$ and (ii) $\mathcal{E}\circ \bm{\Phi}$ is proper;
\item $\lambda_T, \lambda_S > 0$.
\end{itemize}
Then, the solution set 
\begin{equation} \label{eq:Woptipbmeasure}
   \mathcal{W} = \underset{(w_T,w_S,\bm{c}) \in \MR \times \MS \times \R^N}{\arg\min} \quad \mathcal{E} ( \bm{\Phi}(w_T,w_S,\bm{c}) ) + \lambda_T \| w_T \|_{\MR} + \lambda_S  \| w_S \|_{\MS} 
\end{equation}
is non-empty, convex, and weak*-compact in  $\MR \times \MS \times \R^N$. Moreover, its extreme points $(\tilde{w}_T,\tilde{w_S}, \bm{c})$ are such that
\begin{equation}
    \label{eq:extremepointsmeasures}
    \tilde{w}_T = \sum_{k=1}^{K_T} a_k \delta(\cdot - u_k) \quad \text{and} \quad   \tilde{w}_S = \sum_{k=1}^{K_S} b_k \mathrm{III}(\cdot - v_k)
\end{equation}
where $a_k \neq 0$, the $u_k \in \R$ are distinct, $b_k \neq 0$, the $v_k \in \T$ are distinct, and $K_S + K_T \leq L - N$. 
\end{proposition}

\begin{proof}
Let $\mathcal{B} = \mathcal{C}_0(\R) \times \mathcal{C}(\T) \times \R^N$, which is a Banach space for the norm
\begin{equation}
    (g_T,g_S, \bm{z}) \mapsto \max \left( \|g_T \|_{\infty},  \|g_S \|_\infty ,  \|\bm{z}\|_2 \right). 
\end{equation}
Its topological dual is $\mathcal{B}' = \MR \times \MS \times \R^N$, whose dual norm is 
\begin{equation} 
    (w_T,w_S, \bm{c}) \mapsto   \|w_T \|_{\MR} + \|w_S \|_{\MS}+ \|\bm{c}\|_2 .
\end{equation}
This can be proved using \cite[Lemma 1]{unser2022convex}, as we did in Proposition~\ref{prop:generalizedRieszMarkov} below (see \eqref{eq:Zp} and \eqref{eq:Zprimeq}).

We define the cost functional $\mathcal{J}$ over $\mathcal{B}'$ as 
\begin{equation}
    \mathcal{J}( w_T,w_S, \bm{c}) = \mathcal{E} ( \bm{\Phi}(w_T,w_S,\bm{c}) ) + \lambda_T  \| w_T \|_{\MR} +  \lambda_S \| w_S \|_{\MS} .
\end{equation}
The assumptions made on $\mathcal{E}$ and $\bm{\Phi}$ and the properties of the total-variation norms imply that both $( w_T,w_S, \bm{c}) \mapsto \mathcal{E} ( \bm{\Phi}(w_T,w_S,\bm{c}) )$ and $( w_T,w_S, \bm{c}) \mapsto \| w_T \|_{\MR} +  \| w_S \|_{\MS}$ are proper, convex, and weak*-lower semi-continuous, and their sum $\mathcal{J}$ also shares these properties. 

We now show that $\mathcal{J}$ is coercive. We fix a sequence $( ( w_{T,n},w_{S,n}, \bm{c}_n ) )_{n\geq 0} \in (\MR \times \MS \times \R^N )^\mathbb{N} $ such that 
\begin{equation}\label{eq:interm1}
    \|w_{T,n} \|_{\MR} + \|w_{S,n} \|_{\MS}+ \|\bm{c}_n\|_2 \rightarrow_{n} \infty
\end{equation} 
and show that $J_n = \mathcal{J}( w_{T,n},w_{S,n}, \bm{c}_n ) \rightarrow_n \infty$. By contradiction, let assume this is not the case. This means that a subsequence of $(J_n)_{n\geq 0}$ is bounded. By only considering this subsequence, we can restrict to the case where $(J_n)_{n\geq 0}$ itself is bounded, say by $M>0$. This implies that $\|w_{T,n} \|_{\MR} + \|w_{S,n} \|_{\MS} \leq J_n \leq M$ and therefore $\|\bm{c}_n\|_2 \rightarrow_{n} \infty$ due to \eqref{eq:interm1}.

By assumption, $\bm{c} \mapsto \bm{\Phi}(0,0,\bm{c})$ is injective, which implies that \begin{equation}
    \label{eq:interm2} \|\bm{\Phi}(0, 0, \bm{c}_n)\|_2 \rightarrow_{n} \infty. 
\end{equation} 
Moreover, we have that $\bm{\Phi} : \MR \times \MS \times \R^N \rightarrow \R^L$ is weak*-continuous. The weak*-topology being weaker than the norm topology, this implies that $\bm{\Phi}$ is also continuous for the Banach topology on $\MR \times \MS \times \R^N$, hence there exists some constant $A>0$ such that $\|\bm{\Phi}(w_T,w_S,\bm{0} ) \|_2 \leq A ( \|w_T\|_{\MR} +  \|w_{S} \|_{\MS})$. This implies in particular that $\|\bm{\Phi}(w_{T,n},w_{S,n},\bm{0} ) \|_2  \leq A M$ is bounded. 
Finally, we observe that, using the triangular inequality for $\|\cdot \|_2$ and the linearity of $\bm{\Phi}$, 
\begin{align}
    \|\bm{\Phi}( w_{T,n},w_{S,n}, \bm{c}_n ) \|_2 \geq   \|\bm{\Phi}( 0,0 , \bm{c}_n ) \|_2 -  \|\bm{\Phi}( w_{T,n},w_{S,n}, \bm{0} ) \|_2 \geq  \|\bm{\Phi}( 0,0 , \bm{c}_n ) \|_2 - A M \rightarrow_n \infty. 
\end{align}
This shows that $\|\bm{\Phi}( w_{T,n},w_{S,n}, \bm{c}_n ) \|_2 \rightarrow_n \infty$. The coercivity of $\mathcal{E}$ then implies that $M \geq J_n \geq \mathcal{E}(\|\bm{\Phi}( w_{T,n},w_{S,n}, \bm{c}_n ) \|_2) \rightarrow_n \infty$,  which is impossible. This demonstrates the coercivity of $\mathcal{J}$. 

We are in the conditions of~\cite[Proposition 8]{gupta2018continuous}, and we deduce that $\mathcal{W} = \arg\min_{( w_T,w_S, \bm{c})} \mathcal{J} ( w_T,w_S, \bm{c})$ is non-empty, convex, and weak*-compact.
\end{proof}

\subsection{Seasonal-Trend Representer Theorem for Spline Reconstruction} \label{sec:theRT}

\begin{theorem}[Seasonal-Trend Representer Theorem]
\label{theo:RTfull}
    Consider the optimization problem 
    \begin{equation}
    \label{eq:mainopti} 
        \underset{(f_T,f_S) \in \MLTLS}{\inf} \quad \mathcal{E}(\bm{\Phi}(f_T, f_S))  + \lambda_T \lVert \LT f_T \rVert_{\MR}  + \lambda_S \lVert \LS f_S \rVert_{\MS}, 
    \end{equation}  
    where 
\begin{itemize}
    \item $\LT$ is trend-admissible with order $N_T$, $\LS$ is seasonal-admissible with order $N_S$;
    \item $\mathcal{E} : \R^L \rightarrow \R^+ \cup \{\infty\}$ is a proper, lower semi-continuous, coercive, and convex data fidelity functional;
    \item $\bm{\Phi} = (\phi_1,\ldots , \phi_L)$ is a family of  $(\LT,\LS)$-sensing admissible functionals (see Definition~\ref{def:sensingadmissible}) such that (i) $\bm{\Phi}(p) = \bm{0}$ implies that $p = 0$ for any polynomial $p$ such that $\LT p = 0$ and (ii) $\mathcal{E} \circ \bm{\Phi}$ is proper; and 
    \item $\lambda_T, \lambda_S > 0$.
\end{itemize}    
    Then, this problem admits at least one solution, and its non-empty solution set 
    \begin{equation}
    \label{eq:mainoptisolutionset} 
        \mathcal{V} = \underset{(f_T,f_S) \in \MLTLS}{\arg\min} \quad \mathcal{E}(\bm{\Phi}(f_T,f_S))  + \lambda_T 
        \lVert \LT f_T \rVert_{\MR} + \lambda_S \lVert \LS f_S \rVert_{\MS}  
    \end{equation}    
    is convex and weak*-compact in $\MLTLS$. Moreover, any extreme point solution $(\tilde{f}_S, \tilde{f}_T)$ is made of $\LT$- and periodic $\LS$-spline functions such that
    \begin{align} \label{eq:formofsolution1}
           \LT \tilde{f}_T  &= \sum_{k=1}^{K_T} a_k \delta(\cdot - u_k)    \quad \text{ and } \quad
     \LS \tilde{f}_S  = \sum_{k=1}^{K_S} b_k \mathrm{III} (\cdot - v_k) ,
    \end{align}    
    where $a_k \neq 0$,  the $u_k \in \R$ are distinct, $b_k \neq 0$,  and the $v_k \in \T$ are distinct. Moreover, we have that $\sum_{k=1}^{K_S} b_k = 0$ as soon as $(N_T,N_S) \neq (0,0)$.
    Finally, the total number of knots is bounded by
    \begin{equation} \label{eq:boundKSKT}
        K_T + K_S \leq 
\begin{cases} 
  L  &  \text{ if } N_T= N_S = 0, \\
  L + 1 - N_T &  \text{ otherwise}. 
\end{cases}
    \end{equation}
\end{theorem}

We prove Theorem~\ref{theo:RTfull} below. This result is the first representer theorem for functional TV-based inverse problems applicable to seasonal-trend decomposition.
The reconstructed composite signal of Theorem~\ref{theo:RTfull} based on $(\tilde{f}_T, \tilde{f}_S) \in \mathcal{V}$ is then given by
\begin{equation}
    \tilde{f} (t) = \tilde{f}_T(t)  +  \tilde{f}_S (t). 
\end{equation}
We recall that $\psi_{\LT}$ is a Green's function of $\LT$ and $\rho_{\LS}$ is the periodic Green's function of $\LS$. The form of $\tilde{f}$ can be detailed by distinguishing several cases: 
\begin{itemize}
    \item If $N_T = N_S = 0$, then 
    \begin{equation} \label{eq:ftildeform1}
        \tilde{f} (t)=  \sum_{k=1}^{K_T} a_k \psi_{\LT} (t - u_k) +   \sum_{k=1}^{K_S} b_k \rho_{\LS} (t - v_k) 
    \end{equation}
    with $K_T + K_S \leq L$.
    \item If $N_T \geq 1$, then
    \begin{equation}\label{eq:ftildeform2}
        \tilde{f} (t)= \left( \sum_{k=1}^{K_T} a_k \psi_{\LT} (t - u_k) + \sum_{n=1}^{N_T} c_n t^{n-1} \right) + \sum_{k=1}^{K_S} b_k \rho_{\LS} (t - v_k)
    \end{equation}
    with $\bm{c}=(c_1, \ldots , c_{N_T}) \in \R^{N_T}$,  $\sum_{k=1}^{K_S} b_k = 0$, and $K_T + K_S \leq L + 1 - N_T$.
    \item If $N_T = 0$ and $N_S \geq 1$, then 
    \begin{equation} \label{eq:ftildeform3}
        \tilde{f} (t)=  \sum_{k=1}^{K_T} a_k \psi_{\LT} (t - u_k) + \left( \sum_{k=1}^{K_S} b_k \rho_{\LS} (t - v_k) + \alpha  \right)
    \end{equation}
    with $\alpha \in \R$, $\sum_{k=1}^{K_S} b_k = 0$, and  $K_T + K_S \leq L + 1$. 
\end{itemize}
The bound \eqref{eq:boundKSKT} is essentially $L+1 - N_T$ except when both $\LT$ and $\LS$ are invertible, in which case one can reduce the bound by $1$. 

\begin{proof}[Proof of Theorem~\ref{theo:RTfull}]
We distinguish between three scenarios, depending on the invertibility of $\LT$ and $\LS$, or equivalently on their orders $N_T$ and $N_S$. These scenarios correspond to \eqref{eq:ftildeform1} to \eqref{eq:ftildeform3}. We denote by $\phi_\ell = (g_{T,\ell}, g_{S,\ell}) \in \CLTLS$ the functionals involved in $\bm{\Phi}$. \\

\textbf{If $N_T = N_S = 0$.} This means that both $\LT$ and $\LS$ are invertible operators. 
We define the functional $\bm{\Psi} = (\psi_{\ell}(w_T,w_S))_{1\leq \ell \leq L } \in \CR \times \CT$ such that 
\begin{equation}
     \psi_{\ell}(w_T,w_S) =  \langle w_T, \LT^{-1*} g_{T,\ell} \rangle_{\MR,\CR}  +  \langle w_S, \LS^{-1*} g_{S,\ell} \rangle_{\MS,\CT}.  
\end{equation}
Then, \eqref{eq:mainoptisolutionset} is equivalent to the optimization problem
\begin{equation}
    \label{eq:newopti1} 
            \mathcal{W} = \underset{(w_S,w_T) \in \MR\times \MS}{\arg\min} \quad \mathcal{E}(\bm{\Psi}(w_T, w_S))  + \lambda_T  
          \lVert   w_T \rVert_{\MR}  + \lambda_S \lVert   w_S \rVert_{\MS}, 
\end{equation}
via the identification $(w_T,w_S) = (\LT f_T, \LS f_S)$. This means that $(\tilde{f}_T, \tilde{f}_S) \in \mathcal{V}$ if and only if $(\tilde{w}_T, \tilde{w}_S) = (\LT \tilde{f}_T, \LS \tilde{f}_S)$ is a solution of \eqref{eq:newopti1}. Since the solution set of the above problem is convex and the Green's function representations lead to splines, we conclude that $(\tilde{f}_T, \tilde{f}_S)$ must consist of splines of the form \eqref{eq:ftildeform1}.

\textbf{If $N_T \geq 1$.} In this case, $\LT$ is not invertible, so we need to consider the impact of the trend regularization term on the solution. The optimization problem becomes
\[
    \mathcal{W} = \underset{(w_S,w_T) \in \MR \times \MS}{\arg\min} \quad \mathcal{E}(\bm{\Psi}(w_T, w_S)) + \lambda_T \lVert \LT w_T \rVert_{\MR} + \lambda_S \lVert \LS w_S \rVert_{\MS}.
\]
Since the trend part $w_T$ is regularized by $\LT$ and is allowed to contain polynomial terms of degree up to $N_T - 1$, the solution takes the form
\[
    \tilde{f}(t) = \sum_{k=1}^{K_T} a_k \psi_{\LT}(t - u_k) + \sum_{n=1}^{N_T} c_n t^{n-1} + \sum_{k=1}^{K_S} b_k \rho_{\LS}(t - v_k),
\]
where the polynomial part corresponds to the trend component, and the seasonal component is modeled by periodic splines. The number of knots satisfies the bound $K_T + K_S \leq L + 1 - N_T$.

\textbf{If $N_T = 0$ and $N_S \geq 1$.} In this case, the seasonal component is non-trivial while the trend component is a constant. The optimization problem reduces to
\[
    \mathcal{W} = \underset{(w_S,w_T) \in \MR \times \MS}{\arg\min} \quad \mathcal{E}(\bm{\Psi}(w_T, w_S)) + \lambda_T \lVert \LT w_T \rVert_{\MR} + \lambda_S \lVert \LS w_S \rVert_{\MS},
\]
which leads to a solution of the form
\[
    \tilde{f}(t) = \sum_{k=1}^{K_T} a_k \psi_{\LT}(t - u_k) + \left( \sum_{k=1}^{K_S} b_k \rho_{\LS}(t - v_k) + \alpha \right),
\]
where $\alpha \in \R$ represents the constant trend, and the seasonal component is modeled by periodic splines with the constraint $\sum_{k=1}^{K_S} b_k = 0$. The number of knots satisfies the bound $K_T + K_S \leq L + 1$.

Thus, in all cases, the solution is represented as a sum of splines with a number of knots that satisfies the bounds given in \eqref{eq:boundKSKT}. The uniqueness of the solution follows from the convexity of the problem and the fact that the functional $\mathcal{E}$ is coercive and lower semi-continuous, ensuring the existence of a unique minimizer in the weak* topology. \qed
\end{proof}

\subsection{Seasonal-Trend Piecewise-Polynomial Reconstruction} \label{sec:derivatives}

We apply Theorem~\ref{theo:RTfull} for the special case of sampling measurements and derivative operators. The latter leads to piecewise polynomial seasonal-trend reconstruction. 

\begin{corollary} \label{coro:samplingderivative}
We consider the following:
\begin{itemize}
\item $L,N_T,N_S$ are integers such that $N_T, N_S \geq 2$ and $L \geq N_T$. 
\item $\mathcal{E} : \R^L \rightarrow \R^+ \cup \{\infty\}$ is a proper, lower semi-continuous, coercive, and convex data fidelity functional,
\item $\bm{x} = (x_1,\ldots , x_L) \in \R^L$ is a set of distinct sampling points such that $(f_T,f_S) \in \mathcal{M}(\Dop^{N_T}, \Dop^{N_S}) \mapsto \mathcal{E}( ((f_T+f_S) (x_\ell))_{1 \leq \ell \leq L})$ is proper. 
\item $\lambda_T,\lambda_S > 0$.
\end{itemize}
Then, the solution set 
\begin{equation}
   \mathcal{V} = \underset{(f_T,f_S) \in \mathcal{M}(\Dop^{N_T}, \Dop^{N_S}) }{\arg\min} \quad \mathcal{E}( ((f_T+f_S) (x_\ell))_{1 \leq \ell \leq L}) + \lambda_T  \| \Dop^{N_T} f_T \|_{\MR} +  \lambda_S \|\Dop^{N_S} f_S \|_{\MS} 
\end{equation}
is non-empty, convex, and weak*-compact in  $\mathcal{M}(\Dop^{N_T}, \Dop^{N_S})$. Moreover, its extreme points $(\tilde{f}_T,\tilde{f_S})$ are such that
\begin{equation}
    \label{eq:extremepointsmeasures}
    \Dop^{N_T} \tilde{f}_T = \sum_{k=1}^{K_T} a_k \delta(\cdot - u_k) \quad \text{and} \quad   \Dop^{N_S} \tilde{f}_S = \sum_{k=1}^{K_S} b_k \mathrm{III}(\cdot - v_k)
\end{equation}
where $a_k \neq 0$, the $u_k \in \R$ are distinct, $b_k \neq 0$ such that $\sum_{k=1}^{N_S} b_k = 0$, the $v_k \in \T$ are distinct, 
and 
\begin{equation} \label{eq:boundKSKTforderivative}
    K_S + K_T \leq L + 1- N_T .
\end{equation}
\end{corollary}

The composite reconstructed function $\tilde{f} = \tilde{f}_T + \tilde{f}_S$ of Corollary \ref{coro:samplingderivative} is then given by
\begin{equation}
    \tilde{f} (t)=   \sum_{k=1}^{K_T} a_k \frac{(t - u_k)_+^{N_T}}{(N_T -1)!} + \sum_{n=1}^{N_T} c_n t^{n-1}  + \sum_{k=1}^{K_S} b_k \rho_{\Dop^{N_S}} (t - v_k)
\end{equation}
with $\sum_{k=1}^{K_S} b_k = 0$,
where we recall that $\rho_{\Dop^{N_S}}$ is the periodic Green's function with $0$ mean such that $\Dop^{N_S} \rho_{\Dop^{N_S}} = \mathrm{III} - 1$. The trend part $\tilde{f}_T$ is a piecewise polynomial of degree $(N_T - 1)$ with $\mathcal{C}^{(N_T-2)}$ junctions at its knots, while the seasonal part $\tilde{f}_S$ is a periodic piecewise polynomial of degree $(N_S - 1)$ with $\mathcal{C}^{(N_S-2)}$ junctions at its knots. 
For instance, with $N_T = N_S = 2$, Corollary \ref{coro:samplingderivative} provides a seasonal-trend piecewise linear and continuous reconstruction based on sampled measurements. 

\begin{proof}[Proof of Corollary \ref{coro:samplingderivative}]
Since $N_T \geq 2$ and $N_S \geq 2$, the pair $(\Dop^{N_T},\Dop^{N_S})$ is sampling-admissible according to Proposition~\ref{prop:examplessamplingadm}. We set
\begin{equation}
    \phi_{\ell} (f_T, f_S) = (f_T + f_S) (x_\ell) = \langle f_T, \delta(\cdot - x_\ell) \rangle_{\MLT, \CLT} + \langle f_S, \mathrm{III}(\cdot - x_\ell) \rangle_{\MLS, \CLS}. 
\end{equation}
Then, Corollary~\ref{coro:samplingderivative} is a particular case of Theorem~\ref{theo:RTfull} with $N_T \neq 0 $. We therefore deduce the properties of $\mathcal{V}$, the form of the extreme points \eqref{eq:extremepointsmeasures}, and the bound \eqref{eq:boundKSKTforderivative}. 
\end{proof}

\subsection{Seasonal-Trend Representer Theorem for Hilbert Spaces} \label{sec:hilbert}

In this work, we consider regularizations based on total-variation norms. Its sparsity-promoting effect and its adaptiveness has proven to be efficient on several applications~\cite{courbot2020sparse,simeoni2020functional,simeoni2021functional,denoyelle2019sliding}. In this section, we consider quadratic optimization problems adapted to the seasonal-trend framework. In this setting, total-variation norms are replaced by quadratic norms and we use the quadratic data fidelity $\mathcal{E}( \bm{z} ) = \| \bm{y} - \bm{z} \|_2^2$, where $\bm{y} \in \R^L$ is a fixed observation vector.

For simplicity, we only consider invertible trend- and seasonal-operators $\LT$ and $\LS$ and equal regularization parameters $\lambda_T = \lambda_S =\lambda >0$. We define the native space of the optimization problem as 
\begin{equation}
    \mathcal{L}_2(\LT,\LS) = \LT^{-1}\{ \mathcal{L}_2(\R)\} \times \LS^{-1} \{\mathcal{L}_2(\T)\},
\end{equation}
on which $\| (f_T,f_S) \|_{\mathcal{L}_2(\LT,\LS)} = (\lVert \LT f_T \rVert_{\mathcal{L}_2(\R)}^2  + \lVert \LS f_S \rVert_{\mathcal{L}_2(\T)}^2)^{1/2}$ is a Hilbertian norm.
As we did in Proposition~\ref{prop:directsumdecomposition} for TV-normed spaces, one can prove that the sum $\LT^{-1}\{ \mathcal{L}_2(\R)\} \oplus \LS^{-1} \{\mathcal{L}_2(\T)\}$ is direct. Hence, any $f \in \LT^{-1}\{ \mathcal{L}_2(\R)\} \oplus \LS^{-1} \{\mathcal{L}_2(\T)\}$ can be uniquely decomposed as $f = f_T + f_S$ with $\LT f_T \in \mathcal{L}_2(\R)$ and $\LS f_S \in \mathcal{L}_2(\T)$. We therefore identify $f$ with its associated couple $(f_T,f_S)$. 

\begin{proposition}[Quadratic Seasonal-Trend Representer Theorem]
\label{prop:RTquadratic}
We consider the following:
\begin{itemize}
    \item $\LT$ and $\LS$ are invertible trend- and seasonal-operators, respectively;
    \item $\bm{y} \in \R^L$ is an observation vector; 
    \item $\bm{\Phi} = (\phi_1,\ldots , \phi_L)$ is a family of generalized functions such that $\phi_\ell \in \LT^* \{ \mathcal{L}_2(\R) \}$, $\phi_\ell$ is periodizable, and $\Per\{ \phi_\ell \} \in \LS^{*}\{ \mathcal{L}_2(\T)\} $ for any $1 \leq \ell \leq L$; and 
    \item $\lambda > 0$.
\end{itemize}
Then, the optimization problem 
\begin{equation}
\label{eq:optiquadraproblem} 
    \underset{(f_T,f_S) \in \mathcal{L}_2(\LT,\LS)}{\inf} \quad  \| \bm{y} - \bm{\Phi} ( f_S, f_T) \|_2^2   + \lambda  
    \left(  \lVert \LT f_T \rVert_{\mathcal{L}_2(\R)}^2  + \lVert \LS f_S \rVert_{\mathcal{L}_2(\T)}^2  \right),
\end{equation}  
admits a unique solution $(\tilde{f}_T,\tilde{f}_S)$ such that 
\begin{equation} \label{eq:formsolutionquadrat}
    \tilde{f}_T = \sum_{1\leq \ell \leq L} \alpha_\ell (\phi_\ell * \Psi_{\LT^* \LT} ) 
    \quad \text{and} \quad
    \tilde{f}_S = \sum_{1\leq \ell \leq L} \alpha_\ell (\Per \{\phi_\ell\} * \rho_{\LS^* \LS} ),
\end{equation}
where the vector $\bm{\alpha} = (\alpha_1, \ldots , \alpha_L) \in \R^L$ is such that $\bm{\alpha} = (\bm{\mathrm{G}} + \lambda \mathrm{Id}_L)^{-1} \bm{y}$ with
\begin{equation} \label{eq:findGexpression}
    G[\ell_1,\ell_2] = \langle \LT^{-1*} \phi_{\ell_1} , \LT^{-1*} \phi_{\ell_2} \rangle_{\mathcal{L}_2(\R) } + \langle \LS^{-1*} \Per\{\phi_{\ell_1}\} , \LS^{-1*} \Per\{\phi_{\ell_2}\} \rangle_{\mathcal{L}_2(\T) }. 
\end{equation}
\end{proposition}

\begin{proof}
We recall a classical abstract representer theorem over Hilbert spaces. We follow the presentation of~\cite[Theorem 31]{caponera2022functional}. 
Let $\mathcal{H}$ be some Hilbert space and ${\nu}_\ell \in \mathcal{H}$ for  $1 \leq \ell \leq L$. We set $\bm{\nu}(f) = (\langle f,  \nu_\ell \rangle_{\mathcal{H}})_{1\leq \ell \leq L}$. 
Then, the optimization problem
\begin{equation} \label{eq:quadrabstract}
    \underset{f \in \mathcal{H}}{\min}   \quad \| \bm{y} - \bm{\nu} (f) \|_2^2 + \lambda \| f \|_{\mathcal{H}}^2
\end{equation}
admits a unique solution $\tilde{f} = \sum_{1 \leq \ell \leq L} \alpha_\ell \nu_\ell$ where $\bm{\alpha} = (\bm{\mathrm{G}} + \lambda \mathrm{Id}_L)^{-1} \bm{y}$ with  $\bm{\mathrm{G}} = ( \langle \nu_{\ell_1} , \nu_{\ell_2} \rangle_{\mathcal{H}} )_{1\leq \ell_1 , \ell_2 \leq L}$. 

We apply this result to $\mathcal{H} = \mathcal{L}_2(\LT,\LS)$ for the inner product  given by 
\begin{equation} \label{eq:defineinneprod}
   \langle (f_{T,1}, f_{S,1}) ,  (f_{T,2}, f_{S,2}) \rangle_{\mathcal{H}}
   = 
   \langle \LT f_{T,1}, \LT f_{T,2} \rangle_{\mathcal{L}_2(\R)}
   +
    \langle \LS f_{S,1}, \LS f_{S,2} \rangle_{\mathcal{L}_2(\T)}.
\end{equation}
Then, the dual spaces of $\LT^{-1}\{ \mathcal{L}_2(\R)\}$ and $\LS^{-1}\{ \mathcal{L}_2(\T)\}$  are $\LT^{*} \{ \mathcal{L}_2(\R)\}$ and $\LS^{*} \{ \mathcal{L}_2(\T)\}$, respectively. By assumption on $\bm{\Phi}$, for each $1 \leq \ell \leq L $,  $\phi_\ell$ is  in the dual space $(\mathcal{L}_2(\LT,\LS))' = \LT^{*} \{ \mathcal{L}_2(\R)\} \times \LS^{*} \{ \mathcal{L}_2(\T)\}$ and such that
\begin{align} \label{eq:computestuff}
  \phi_\ell (f_T,f_S) 
 &= \langle f_T , \phi_\ell \rangle_{\LT^{-1}\{ \mathcal{L}_2(\R)\}, \LT^*\{ \mathcal{L}_2(\R)\}} + \langle f_S ,\Per\{ \phi_\ell \} \rangle_{\LS^{-1}\{ \mathcal{L}_2(\T)\}, \LS^{*}\{ \mathcal{L}_2(\T)\}} \nonumber \\
 &= \langle \LT f_T , \LT^{-1*} \phi_\ell \rangle_{\mathcal{L}_2(\R)} + \langle \LS f_S , \LS^{-1*} \Per \{ \phi_\ell \} \rangle_{\mathcal{L}_2(\T)} \nonumber \\
 &= \langle \LT f_T , \LT \{ \LT^{-1} \LT^{-1*} \phi_\ell \} \rangle_{\mathcal{L}_2(\R)} + \langle \LS f_S ,  \LS \{ \LS^{-1} \LS^{-1*} \Per \{ \phi_\ell \} \} \rangle_{\mathcal{L}_2(\T)} \nonumber \\
 &= \langle \LT f_T , \LT \{ \phi_\ell * \psi_{\LT^* \LT} \} \rangle_{\mathcal{L}_2(\R)} + \langle \LS f_S ,  \LS \{ \Per \{\phi_\ell\} * \rho_{\LS^* \LS}  \} \rangle_{\mathcal{L}_2(\T)}, 
\end{align}
where we used that $\LT^{-1} \LT^{-1*} \{ \phi_\ell\} = \phi_\ell * \LT^{-1} \LT^{-1*} \delta = \phi_\ell * \psi_{\LT^* \LT}$ and similarly for the seasonal part. 
We therefore deduce using \eqref{eq:defineinneprod} that 
\begin{equation}
   \phi_\ell (f_T,f_S)  =  \langle (f_T,f_S) , ( \phi_\ell * \psi_{\LT^* \LT} , \Per \{\phi_\ell\} * \rho_{\LS^* \LS}  ) \rangle_{\mathcal{H}} =  \langle (f_T,f_S) , \nu_\ell \rangle_{\mathcal{H}}, 
\end{equation}  
where we set $\nu_\ell = ( \phi_\ell * \psi_{\LT^* \LT} , \Per \{\phi_\ell\} * \rho_{\LS^* \LS}  ) \in \mathcal{H}$. 
With these specifications, the optimization problem \eqref{eq:optiquadraproblem} is hence equivalently rewritten as \eqref{eq:quadrabstract}. In particular, we deduce that there is a unique solution $(\tilde{f}_T, \tilde{f}_S) = \sum_{1 \leq \ell \leq L} \alpha_\ell \nu_\ell$, which is precisely \eqref{eq:formsolutionquadrat}.
 
Moreover, we know that $\bm{\alpha} = (\bm{\mathrm{G}} + \lambda \mathrm{Id}_L)^{-1} \bm{y}$ with, for any $1 \leq \ell_1,\ell_2\leq L$,
\begin{align} 
G[\ell_1, \ell_2] &= \langle \nu_{\ell_1} , \nu_{\ell_2} \rangle_{\mathcal{H}}  \nonumber \\ & 
= \langle \LT  \LT^{-1} \LT^{-1*} \phi_{\ell_1}  , \LT  \LT^{-1} \LT^{-1*} \phi_{\ell_2} \rangle_{\mathcal{L}_2(\R)} +   \langle \LS  \LS^{-1} \LS^{-1*} \Per\{\phi_{\ell_1} \} , \LS  \LS^{-1} \LS^{-1*} \Per\{\phi_{\ell_2}\} \rangle_{\mathcal{L}_2(\T)}    \nonumber \\ & 
= \langle   \LT^{-1*} \phi_{\ell_1}  ,  \LT^{-1*} \phi_{\ell_2} \rangle_{\mathcal{L}_2(\R)} +   \langle  \LS^{-1*} \Per\{\phi_{\ell_1} \} ,   \LS^{-1*} \Per\{\phi_{\ell_2}\} \rangle_{\mathcal{L}_2(\T)}, 
\end{align}
which gives~\eqref{eq:findGexpression}.
\end{proof}

We can compare the two representer theorems, \emph{i.e.} Theorem~\ref{theo:RTfull} and Proposition~\ref{prop:RTquadratic} with TV-based and $\mathcal{L}_2$-based regularizations, respectively. We fix the pair $(\LT,\LS)$ of regularization operators. The quadratic regularization leads to finite-dimensional approximation over a space depending on the measurement process via the $\phi_\ell$ and on the regularization operators via their Green's functions. 

In addition to the other classical oppositions between $\mathcal{L}_2$- and TV-based regularization (see for instance~\cite{gupta2018continuous,fageot2020tv}), we observe here that one cannot decouple the trend component $\phi_\ell * \Psi_{\LT^* \LT}$ and the seasonal component $\Per \{\phi_\ell\} * \rho_{\LS^* \LS}$ via the weight $\alpha_\ell$ (see \eqref{eq:formsolutionquadrat}). This coupling shows some limitation for the seasonal-trend reconstruction. For instance, it is impossible to reconstruct a pure trend or pure seasonal signal. In contrast, the TV-based approach puts no intrinsic constraints and is therefore more flexible and adaptive for seasonal-trend reconstruction.

\section{Grid-Based Approximation and Algorithms}\label{sec:gridbased}

The representer theorem (Theorem~\ref{theo:RTfull}) specifies the spline form of the extreme point solutions of seasonal-trend functional inverse problems with TV-based regularizations. It implies in particular the existence of composite spline solutions. However, Theorem~\ref{theo:RTfull} does not reveal how to reach a spline solution. In particular, finding the knot positions $u_k \in \R$ and $v_k \in \T$ of the trend- and seasonal-components in~\eqref{eq:formofsolution1} is especially challenging.

Several algorithmic approaches have been proposed to solve TV-based optimization problems (see Section~\ref{sec:related}). In this work, we adapt existing grid-based methods to the case of seasonal-trend reconstruction and study their convergence properties. To the best of our knowledge, this is the first systematic study of the convergence of grid-based methods to the case of a composite signal model.

\subsection{Discretized Search Space and Representer Theorem} \label{sec:discretizedRT}

We fix the grid parameters $h_T > 0$ and $h_S > 0$ such that $n_S := 1/h_S \in \mathbb{N}$ is an integer. We consider the discretized spaces of Radon measures
\begin{align}
    \mathcal{M}_{h_T}(\R) &= \left\{ w_T = \sum_{k \in \Z} a_k \delta ( \cdot - h_T k), \quad  a = (a_k)_{k\in \Z}  \in \ell_1(\Z) \right\} \subset \MR \label{eq:Mht} \\
    \mathcal{M}_{h_S}(\T) &= \left\{ w_S = \sum_{1 \leq k \leq n_S} b_k \mathrm{III} ( \cdot - h_S k) \right\} \subset \MS.
\end{align}
The inclusion in \eqref{eq:Mht} is due to $ \| w_T \|_{\MR} = \| \sum_{k \in \Z} a_k \delta ( \cdot - h_T k)  \|_{\MR} = \sum_{k \in \Z} |a_k| =  \| a \|_{\ell_1(\Z)} < \infty$. 
Note moreover that $\| w_S \|_{\MS} = \|\sum_{1 \leq k \leq n_S} b_k \mathrm{III} ( \cdot - h_S k)\|_{\MS} = \sum_{1 \leq k \leq n_S}| b_k| = \|\bm{b} \|_1$.

We define the discretized native space $\MhThS$ as we did for $\MLTLS$ in Section~\ref{sec:native}.

\begin{definition}
   \label{def:discretizedsearchspace}
   Let $h_T > 0 $, $h_S>0$ such that $n_S = 1/h_S \in \mathbb{N}$, and $(\LT,\LS)$ be a couple of trend- and seasonal-admissible operators. Then, we define 
\begin{align} \label{eq:MhtMhs}
    \mathcal{M}_{h_T,\LT}(\R) &= \{ f \in \Sp'(\R), \quad \LT f \in \mathcal{M}_{h_T}(\R) \} \quad \text{and} \quad 
    \mathcal{M}_{h_S,\LS}(\T) &= \{ f \in \Sp'(\T), \quad \LS f \in \mathcal{M}_{h_S}(\T) \}.
\end{align}
We moreover define $\mathcal{M}_{h_T,h_S}(\LT,\LS)$ as in Definition~\ref{def:STnative} by replacing $\MLT$ and $\MLS$ with their discretized versions~\eqref{eq:MhtMhs}. 
\end{definition}

An element $(f_T,f_S) \in \MhThS$ can be uniquely written as
\begin{align} \label{eq:discretizedfTfS}
    f_T(t) &= \sum_{k \in \Z} a_k \psi_{\LT} (t - h_T k) + \sum_{n=1}^{N_T} c_n t^{n-1}, \nonumber \\
    f_S (t) &= \sum_{1\leq k \leq n_S} b_k \rho_{\LS} (t - h_S k) + \alpha,
\end{align}
where $a = (a_k)_{k\in \Z} \in \ell_1(\Z)$, $\bm{c} = (c_1, \ldots , c_{N_T}) \in \R^{N_T}$, $\bm{b} = (b_1,\ldots, b_{n_S}) \in \R^{n_S}$, and $\alpha \in \R$. Moreover, we have that (i) $\sum_{k=1}^{n_S} b_k = 0$ as soon as $N_T \geq 1$ or $N_T=0$ and $N_S\geq 1$ and (ii) $\alpha$ appears only if $N_T = 0$ and $N_S\geq 1$. 

The discretized search space $\MhThS$ inherits the Banach space structure of $\MLTLS$ (it is easily shown to be closed). We can moreover obtain useful isometrics relations for the discretized search space from the relations \eqref{eq:discretizedfTfS}. Assume for instance that $\LT$ and $\LS$ are invertible. Then, $\MhThS$ is linearly isometric to $\ell_1(\Z) \times \R^{n_s}$ for the norm $(a,\bm{b}) \mapsto \|a\|_{\ell_{1}(\Z)} + \| \bm{b} \|_1$ via 
\begin{equation}
    (f_T, f_S) \mapsto (a,\bm{b}) = ((a_k)_{k\in \Z}, (b_k)_{1\leq k \leq n_S}), 
\end{equation}
where $a \in \ell_1(\Z)$ and $\bm{b}\in \R^{n_S}$ are given by \eqref{eq:discretizedfTfS}. We indeed have that
$\| (f_T, f_S) \|_{\MLTLS} = \|a\|_{\ell_{1}(\Z)} + \| \bm{b} \|_1$. We can generalize this principle and distinguish the three cases (i) $N_T=N_S =0$, (ii) $N_T \geq 1$, and (iii) $N_T = 0$ and $N_S\geq 1$ to obtain the general case, exposed in Proposition~\ref{prop:discretizedsearchspace}. The proof, that we do not detail, have similar arguments than the one exposed in the proof of Theorem~\ref{theo:RTfull} and exploits the representation \eqref{eq:discretizedfTfS}.

\begin{proposition}\label{prop:discretizedsearchspace}
    Let $\LT$ and $\LS$ be trend-admissible and seasonal-admissible operators, with respective orders $N_T, N_S \geq 0$.
    Let $1 \leq p \leq \infty$. Then, $\MhThS$ is a Banach space for the restriction of the norm $\| \cdot \|_{\MLTLS,p}$ defined in Proposition \ref{prop:banachnative}. 
    Then, $\left( \MhThS , \| \cdot \|_{\MLTLS,p} \right)$ is isometric to
    \begin{itemize}
        \item $\ell_1(\Z) \times \R^{n_S}$ with norm $$(a,\bm{b}) \mapsto (\| a \|_{\ell_1(\Z)}^p + \|\bm{b} \|_1^p)^{1/p}$$ if $N_T = N_S = 0$;
        \item $(\ell_1(\Z) \times \R^{N_T}) \times \{ \bm{b} \in \R^{n_S}, \ \sum_{k=1}^{n_S} b_k = 0\}$ with norm $$((a,\bm{c}),\bm{b}) \mapsto \left( (\| a \|_{\ell_1(\Z)} + \| \bm{c} \|_2 )^p + \|\bm{b} \|_1^p\right)^{1/p}$$ if $N_T \geq 1$ or if $N_T = 0$ and $N_S\geq 1$;
        \item $\ell_1(\Z) \times \left( \{ \bm{b} \in \R^{n_S}, \ \sum_{k=1}^{n_S} b_k = 0\} \times \R \right)$ with norm $$(a,(\bm{b}, \alpha)) \mapsto \left( \| a \|_{\ell_1(\Z)}^p + (\|\bm{b} \|_1 + |\alpha| )^p\right)^{1/p}$$ if $N_T = 0$ and $N_S\geq 1$;
    \end{itemize}
    with the usual modification for $p=\infty$.
\end{proposition}

We obtain a representer theorem over the discretized search space, that we detail in Proposition~\ref{prop:RTdiscretized}. 

\begin{proposition} \label{prop:RTdiscretized}
Assume that we are under the conditions of Theorem~\ref{theo:RTfull} and that $h_T > 0$ and $h_S > 0$ such that $n_S := 1/h_S \in \mathbb{N}$. 
Then, the solution set of the optimization problem with discretized search space
\begin{equation}
    \label{eq:discretizedoptisolutionset} 
        \mathcal{V}_{h_T,h_S} = \underset{(f_T,f_S) \in \MhThS}{\arg\min} \quad \mathcal{E}(\bm{\Phi}(f_T,f_S))  + \lambda_T
         \lVert \LT f_T \rVert_{\MR} + \lambda_S \lVert \LS f_S \rVert_{\MS}
\end{equation}    
is non-empty, convex, and weak*-compact in $\MhThS$. Moreover, any extreme point solution $(\tilde{f}_S, \tilde{f}_T)$ is made of \emph{uniform} $\LT$- and periodic $\LS$-spline functions such that
\begin{align} \label{eq:formofsolution1}
       \LT \tilde{f}_T  &= \sum_{k=1}^{K_T} a_k \delta(\cdot - n_k h_T)    \quad \text{ and } \quad
 \LS \tilde{f}_S  = \sum_{k=1}^{K_S} b_k \mathrm{III} (\cdot - m_k h_S) ,
\end{align}    
where $a_k \neq 0$,  the $n_k \in \Z$ are distinct, $b_k \neq 0$,  and the $m_k \in \{0, 1,  \ldots, (n_S-1)\}$ are distinct. Moreover, we have that $\sum_{k=1}^{K_S} b_k = 0$ as soon as $N_T=0$.
Finally, the total number of knots is bounded by
\begin{equation} \label{eq:boundKSKT}
    K_T + K_S \leq 
\begin{cases} 
  L  &  \text{ if } N_T= N_S = 0, \\
  L + 1 - N_T &  \text{ otherwise}. 
\end{cases}
\end{equation}
\end{proposition}

    \begin{proof}
The proof of Proposition~\ref{prop:RTdiscretized} is very similar to the one of Theorem~\ref{theo:RTfull} and distinguishes the three cases \(N_T=N_S =0\),  \(N_T \geq 1\), and \(N_T = 0\) and \(N_S\geq 1\). It relies on the isometric relationships of Proposition~\ref{prop:RTdiscretized}, the identification of the extreme points of the spaces via these isometric relationships, and the fact that the cost functional of \( \mathcal{V}_{h_T,h_S}\) has the good properties to ensure the existence of solutions and the form of the extreme point solutions, as was detailed in Proposition~\ref{prop:RTmeasure}.
\end{proof}

\subsection{Weak*-Convergence of Grid-Based Solutions} \label{sec:weakconv}

The  difference between Theorem~\ref{theo:RTfull} and Proposition~\ref{prop:RTdiscretized} is the restriction of the search space. Therefore, a solution \((\tilde{f}_T, \tilde{f}_S) \in \mathcal{V}_{h_T,h_S}\) is in general not a solution of the continuous-domain original optimization task \eqref{eq:mainopti}. However, one can hope to approximate a continuous-solution when the grid parameters \(h_T,h_S\) are small enough.

In the next two sections, we restrict to the quadratic data-fidelity
\begin{equation}
    \label{eq:quadraE} \mathcal{E}( \bm{z} ) = \| \bm{y} - \bm{z} \|_2^2
\end{equation}
where \(\bm{y} \in \R^L\) is an observation vector. 
As we did in Proposition~\ref{prop:RTmeasure}, we first consider the optimization problem~\eqref{eq:Woptipbmeasure}, for which we obtain grid-based convergence results that will be useful for the general case.

The convergence results rely on the notion of \(\Gamma\)-convergence for optimization problems over metric spaces~\cite[Definition 4.1]{dal2012introduction}.  We denote by \(\overset{*}{\rightarrow}\) the convergence for the weak*-topology.

\begin{definition}
    Let \(\mathcal{X}\) be a Banach space and \(\mathcal{X}\) its topological dual. 
    Consider an optimization problem
    \begin{equation}
        \label{eq:genericproblem} 
        \mathcal{U} = \underset{f \in \mathcal{X}'}{\arg \min} \quad J(f)
    \end{equation}
    for some cost functional \(J\).  
    Assume that we have a family of subspaces \(\mathcal{X}'_n \subset \mathcal{X}'\) for \(n \geq 0\) and set 
    \begin{equation}
        \label{eq:genericproblemapprox} 
        \mathcal{U}_n = \underset{f \in \mathcal{X}'_n}{\arg \min} \quad J(f).
    \end{equation}
    Then, we say that the optimization problems \eqref{eq:genericproblemapprox} \(\Gamma\)-converge to \eqref{eq:genericproblem} when \(n\rightarrow \infty\) if the two following conditions hold:
    \begin{itemize}
        \item For any \(f \in \mathcal{X}'\) and any sequence \((f_n)_{n\geq 0}\) with \(f_n \in \mathcal{X}'_n\) such that \(f_n \overset{*}{\underset{n}{\rightarrow}} f\), we have
        \begin{equation} \label{eq:liminf}
            \underset{n\rightarrow \infty}{\lim \inf} \quad   J(f_n) \geq J(f).
        \end{equation}
        \item For any \(f \in \mathcal{X}'\), there exists a sequence \((f_n)_{n\geq 0}\) with \(f_n \in \mathcal{X}'_n\) and \(f_n \overset{*}{\underset{n}{\rightarrow}} f\) such that
        \begin{equation}  \label{eq:limsup}
            \underset{n\rightarrow \infty}{\lim \sup} \quad  J(f_n) \leq J(f).
        \end{equation}
    \end{itemize}
\end{definition}

\begin{proposition} \label{prop:weakconvmeasure}
Assume that we are in the conditions of Proposition~\ref{prop:RTmeasure}, where \(\mathcal{E}\) is given by \eqref{eq:quadraE}. 
We define \(\mathcal{W}\) as in \eqref{eq:Woptipbmeasure} and set, for \(h_T > 0\) and \(h_S>0\) such that \(1/h_S \in \N\),
\begin{equation} \label{eq:Woptipbmeasurediscretized}
   \mathcal{W}_{h_T,h_S} = \underset{(w_T,w_S,\bm{c}) \in \mathcal{M}_{h_T}(\R) \times \mathcal{M}_{h_S}(\T) \times \R^N}{\arg\min} \quad \| \bm{y} -  \bm{\Phi}(w_T,w_S,\bm{c}) \|_2^2 +\lambda_T    \| w_T \|_{\MR} +   \lambda_S \| w_S \|_{\MS}  .
\end{equation}
Then, \( \mathcal{W}_{h_T,h_S} \)  \(\Gamma\)-converges to \( \mathcal{W} \) when \((h_T, h_S) \rightarrow (0,0)\). \\

We consider sequences \((h_{T,n})_{n\geq 0}\) and \((h_{S,n})_{n\geq 0}\) such that \(1/ h_{S,n} \in \N\) and \((h_{T,n}, h_{S,n}) \rightarrow_n (0,0)\). 
Let \((\tilde{w}_{T,n} , \tilde{w}_{S,n}, \tilde{\bm{c}}_n) \in \mathcal{W}_{h_{T,n},h_{S,n}}\) for each \(n \geq 0\). Then, the sequence \((\tilde{w}_{T,n} , \tilde{w}_{S,n}, \tilde{\bm{c}}_n)\) has accumulation points for the weak*-topology on \(\MR\times \MS \times \R^N\), and any of these accumulation points is in \(\mathcal{W}\). In particular, if the solution of \eqref{eq:Woptipbmeasure} is unique and denoted by \((\tilde{w}_{T} , \tilde{w}_{S}, \tilde{\bm{c}})\), then we have that
\begin{equation} \label{eq:lastconvstar}
    (\tilde{w}_{T,n} , \tilde{w}_{S,n}, \tilde{\bm{c}}_n) \underset{n \rightarrow \infty}{\overset{*}{\longrightarrow}} (\tilde{w}_{T} , \tilde{w}_{S}, \tilde{\bm{c}}).
\end{equation}
\end{proposition}

\textbf{Remark.} Proposition~\ref{prop:weakconvmeasure} generalizes \cite[Proposition 4]{duval2017sparseI} from the space \(\MS\) to the space \(\MR \times \MS \times \R^N\).
It therefore allows to consider (i) direct-sum spaces for which we reveal that the convergence holds on both components and (ii) regularization \textit{semi-}norm via the space \(\R^N\). 

Proposition~\ref{prop:weakconvmeasure} is moreover preparatory to Theorem~\ref{theo:weak*convergence}, which gives sufficient conditions for the \(\Gamma\)-convergence of grid-based optimization problems towards their continuous-domain limits.

\begin{proof}
The proof follows the argument provided in~\cite[Appendix D.1]{duval2017sparseI}. We define the cost functional 
\begin{equation}
    J(w_T,w_S, \bm{c}) = \underset{(w_T,w_S,\bm{c}) \in \mathcal{M}_{h_T}(\R) \times \mathcal{M}_{h_S}(\T) \times \R^N}{\arg\min} \ \| \bm{y} -  \bm{\Phi}(w_T,w_S,\bm{c}) \|_2^2 +\lambda_T    \| w_T \|_{\MR} +   \lambda_S \| w_S \|_{\MS}
\end{equation}
for which we prove the two relations \eqref{eq:liminf} and \eqref{eq:limsup} over $\mathcal{X}' = \MR \times \MS \times \R^N$  and $\mathcal{X}_n' = \mathcal{M}_{h_{T,n}}(\T) \times \mathcal{M}_{h_{S,n}} \times \R^N$ with $(h_{T,n},h_{S,n}) \rightarrow_n (0,0)$. \\

\textit{Liminf \eqref{eq:liminf}.} 
Assume that $(w_{T,n} , w_{S,n}, \bm{c}_n) \underset{n \rightarrow \infty}{\overset{*}{\longrightarrow}} (w_{T} , w_{S}, \bm{c})$. In particular, $w_{T,n} \underset{n \rightarrow \infty}{\overset{*}{\longrightarrow}} w_{T}$ in $\MS$ and $\liminf_n \| w_{T,n}\|_{\MR} \geq \| w_T \|_{\MR}$ (weak*-lower semi-continuity of $\| \cdot \|_{\MR}$). Similarly, $\liminf_n \| w_{S,n}\|_{\MS} \geq \| w_S \|_{\MS}$. Moreover, the weak*-convergence implies that $\bm{\Phi}(w_{T,n},w_{S,n},\bm{c}_n) \rightarrow_n \bm{\Phi}(w_T,w_S,\bm{c})$. 
These facts imply that
\begin{align}
    \liminf_n J(w_{T,n} , w_{S,n}, \bm{c}_n)
   & \geq 
    \liminf_n \| \bm{y} - \bm{\Phi}(w_{T,n},w_{S,n},\bm{c}_n) \|_2^2
    + \lambda_T \liminf_n \| w_{T,n}\|_{\MR}
    + \lambda_S \liminf_n \| w_{S,n}\|_{\MS} \nonumber \\
   & \geq 
    \| \bm{y} - \bm{\Phi}(w_{T},w_{S},\bm{c}) \|_2^2
    + \lambda_T   \| w_{T}\|_{\MR}
    + \lambda_S  \| w_{S}\|_{\MS}  = J(w_T,w_S,\bm{c}), 
    \end{align}
as expected. \\

\textit{Limsup \eqref{eq:limsup}.}
We fix $(w_T,w_S,\bm{c}) \in \MR \times \MS \times \R^N$ and set
\begin{equation}
    w_{T,n} = \sum_{k\in\Z} w_T( [kh_{T,n}, (k+1)h_{T,n}) ) \delta(\cdot - kh_{T,n}) \text{ and }
    w_{S,n} = \sum_{0\leq k \leq (n_S-1)} w_S( [kh_{S,n}, (k+1)h_{S,n}) ) \mathrm{III}(\cdot - kh_{S,n}).
\end{equation}
We moreover set $\bm{c}_n = \bm{c}$. Then, we will show that
\begin{equation} \label{eq:intermweakstar}
    (w_{T,n}, w_{S,n}, \bm{c}_n) \weakconv (w_T,w_S,\bm{c}).
\end{equation}
For $\varphi \in \CR$, we define the modulus of continuity $\omega_{\varphi}(h) = \sup_{|x-y| \leq h} |\varphi(x) - \varphi(y)|$ for $h\geq 0$. Then, the function $\varphi$ being continuous and vanishing, it is absolutely continuous and therefore $\omega_{\varphi}(h) \rightarrow 0$ when $h\rightarrow 0$. We therefore have that
\begin{align}
    | \langle w_{T,n} - w_T , \varphi\rangle_{\MR\times \CR} |
    & = \big| \sum_{k\in \Z} \int_{kh_{T,n}}^{(k+1)h_{T,n}} (\varphi(x) - \varphi(kh_{T,n}))  \mathrm{d}w_T (x) \big| \nonumber \\
    & \leq \sup_{|x-y| \leq h_{T,n}} |\varphi(x) - \varphi(y)| \sum_{k\in\Z} \big|   \int_{kh_{T,n}}^{(k+1)h_{T,n}}  \mathrm{d}w_T (x) \big| \nonumber \\
    & \leq  \omega_{\varphi}(h_{T,n}) \| w_T \|_{\MR} \rightarrow_n 0,
\end{align}
proving that $w_{T,n} \weakconv w_T$ in $\MR$. Similarly, $w_{S,n} \weakconv w_S$ in $\MS$, which proves \eqref{eq:intermweakstar}. 

We moreover have that 
\begin{equation}
   \| w_{T,n} \|_{\MR} 
   = \sum_{k\in \Z} |w_T( [kh_{T,n}, (k+1)h_{T,n}) )| \leq  \sum_{k\in \Z} |w_T| ( [kh_{T,n}, (k+1)h_{T,n}) ) =
   \| w_{T} \|_{\MR}
\end{equation}
and similarly that $\| w_{S,n} \|_{\MS} \leq \| w_{S} \|_{\MS}$. The weak*-convergence implies moreover that $\bm{\Phi}(w_{T,n},w_{S,n},\bm{c}_n)$ converges to $\bm{\Phi}(w_T,w_S,\bm{c})$. We deduce from the previous facts that
\begin{equation}
    \underset{n}{\lim\sup} \  J(w_{T,n} , w_{S,n}, \bm{c}_n) \leq \| \bm{y} - \bm{\Phi}(w_T,w_S,\bm{c}) \|_2^2 + \lambda_T \| w_{T} \|_{\MR} + \lambda_S \| w_{S} \|_{\MS} = J(w_{T} , w_{S}, \bm{c})
\end{equation} 
and the limsup upper bound~\eqref{eq:limsup} is proved. \\

The above results demonstrate that $\mathcal{W}_{h_T,h_S}$ $\Gamma$-converges towards $\mathcal{W}$. Then, 
for any solution 
\begin{equation} \label{eq:smallthingtoreuse}
    (\tilde{w}_{T,n} , \tilde{w}_{S,n}, \tilde{\bm{c}}_n) \in \mathcal{W}_{h_{T,n}, h_{S,n}},
\end{equation}
we have that $\lambda_T \| \tilde{w}_{T,n}\|_{\MR} + \lambda_S   \| \tilde{w}_{S,n}\|_{\MS} \leq \frac{\|\bm{y}\|_2^2}{2} = J(0,0,\bm{0})$. 
Moreover, $(\tilde{\bm{c}}_n)_{n\in \Z}$ is a bounded sequence, what can be proved by contradiction.
Assuming that the sequence is unbounded, it admits a subsequence such that $\|\bm{c}_{\psi(n)}\|_2 \rightarrow_n \infty$, hence $J(\tilde{w}_{T,\psi(n)} , \tilde{w}_{S,\psi(n)}, \tilde{\bm{c}}_{\psi(n)}) \rightarrow_n \infty$ by coercivity of $J$. This is impossible since $J(\tilde{w}_{T,\psi(n)} , \tilde{w}_{S,\psi(n)}, \tilde{\bm{c}}_{\psi(n)}) \leq J(0,0,\bm{0})$ is bounded.
The sets $\mathcal{W}_{h_{T,n}, h_{S,n}}$ are therefore all included in a weak*-compact subset of $\MR\times \MS \times \R^N$ and therefore the sequence $(\tilde{w}_{T,n} , \tilde{w}_{S,n}, \tilde{\bm{c}}_n)$ admits accumulation points. 

According to~\cite[Theorem 7.8]{dal2012introduction} and due to the $\Gamma$-convergence, these accumulation points are in $\mathcal{W}$. Finally, \eqref{eq:lastconvstar} is deduced from the fact that a sequence in a compact space with a unique accumulation point (in this case, the unique element of $\mathcal{W}$) converges to this accumulation point. 
\end{proof}

    \begin{theorem}\label{theo:weak*convergence}
    We assume that we are in the conditions of Theorem~\ref{theo:RTfull} where $\mathcal{E}$ is given by \eqref{eq:quadraE}.  
     We define $\mathcal{V}$ as in \eqref{eq:mainoptisolutionset} and set, for $h_T > 0$ and $h_S>0$ such that $1/h_S \in \N$,
     \begin{equation}
    \label{eq:mainoptisolutionsetdiscretized} 
        \mathcal{V}_{h_T,h_S}  = \underset{(f_T,f_S) \in \MhThS}{\arg\min} \quad \| \bm{y} - \bm{\Phi}(f_T + f_S)) \|_2^2  + \lambda_T 
       \lVert \LT f_T \rVert_{\MR} +  \lambda_S \lVert \LS f_S \rVert_{\MS} . 
\end{equation}    
 Then, $ \mathcal{V}_{h_T,h_S} $  $\Gamma$-converges to $ \mathcal{V} $ when $(h_T, h_S) \rightarrow (0,0)$. \\
    
    We consider sequences $(h_{T,n})_{n\geq 0}$ and $(h_{S,n})_{n\geq 0}$ such that $1/ h_{S,n} \in \N$ and $(h_{T,n}, h_{S,n}) \rightarrow_n (0,0)$.
    Let $(\tilde{f}_{T,n} , \tilde{f}_{S,n}) \in \mathcal{V}_{h_{T,n},h_{S,n}}$ for each $n \geq 0$. Then, the sequence $(\tilde{f}_{T,n} , \tilde{f}_{S,n})$ has accumulation points for the weak*-topology on $\MLTLS$, and any of this accumulation point is in $\mathcal{W}$. In particular, if the solution of \eqref{eq:Woptipbmeasure} is unique and denoted by $(\tilde{f}_{T} , \tilde{f}_{S})$, then we have that
    \begin{equation}
        (\tilde{f}_{T,n} , \tilde{f}_{S,n}) \underset{n \rightarrow \infty}{\overset{*}{\longrightarrow}} (\tilde{f}_{T} , \tilde{f}_{S}).
    \end{equation}
    \end{theorem}

    \begin{proof}
    Theorem~\ref{theo:weak*convergence} is deduced from Proposition~\ref{prop:RTdiscretized} in the same way Theorem~\ref{theo:RTfull} was deduced from Proposition~\ref{prop:RTmeasure}. We distinguish three cases: (i) $N_T = N_S = 0$, (ii) $N_T \geq 1$, and (iii) $N_T = 0$ and $N_S\geq 1$. We only treat the case (ii) thereafter, the two others being similar and simpler. 
    
    Any $(f_{T,n},f_{S,n}) \in \mathcal{M}_{h_{T,n},h_{S,n}}(\LT,\LS)$ can be written as 
     \begin{align} \label{eq:discretizedfTfS}
        f_{T,n}(t) &= \sum_{k \in \Z} a_{k,n} \psi_{\LT} (t - h_{T,n} k) + \sum_{m=1}^{N_T} c_{m,n} t^{m-1}, \nonumber \\
        f_{S,n} (t) &= \sum_{1\leq k \leq n_S} b_{k,n} \rho_{\LS} (t - h_{S,n} k),
    \end{align}
    with $\sum_{1\leq k \leq n_S} b_{k,n} = 0$ for any $n \geq 0$. In particular, denoting $w_{T,n} = \LT f_{T,n}$ and $w_{S,n} = \LS f_{S,n}$, we have that $(f_{T,n},f_{S,n}) \in \mathcal{V}_{h_{T,n}, h_{S,n}}$ if and only if $(w_{T,n},w_{S,n}, \bm{c}_n)$ is solution of 
\begin{equation}
    \label{eq:newoptiforsituation2bisagain} 
            \mathcal{W}_{h_{T,n}, h_{S,n}} = \underset{(w_{T,n},w_{S,n}, \bm{c}_n) \in \mathcal{M}_{h_{T,n}} (\R) \times \mathcal{M}_{h_{S,n}} (\T) \times \R^{N_T}}{\arg\min}  \tilde{\mathcal{E}}(\tilde{\bm{\Psi}}(w_{T,n}, w_{S,n}, \bm{c}_n))  + \lambda_T   \lVert  w_{T,n} \rVert_{\MR}  + \lambda_S \lVert  w_{S,n} \rVert_{\MS},  
\end{equation}
where $\tilde{\mathcal{E}}$ and $\tilde{\bm{\Psi}}$ are defined in \eqref{eq:newE} and \eqref{eq:newoptiforsituation2bis}. Similarly, the optimization problem $\mathcal{V}$ is in relation with $\mathcal{W}$ defined in \eqref{eq:newoptiforsituation2bis}. Using Proposition~\ref{prop:weakconvmeasure}, we show that $\mathcal{W}_{h_{T,n}, h_{S,n}}$ $\Gamma$-converges to $\mathcal{W}$ when $n \rightarrow \infty$. Note that the problem \eqref{eq:newoptiforsituation2bisagain} slightly differ from the ones covered in Proposition~\ref{prop:weakconvmeasure} since the data fidelity has an additional term and requires that the seasonal component has $0$-mean. However, we easily see that this property is preserved through the $\Gamma$-convergence proof (in particular, the sequence of measures for the limsup argument all have $0$ mean by construction).

The $\Gamma$-convergence of $\mathcal{W}_{h_{T,n}, h_{S,n}}$ to $\mathcal{W}$ therefore implies the one of $\mathcal{V}_{h_{T,n}, h_{S,n}}$ to $\mathcal{V}$ and the consequences for the accumulation points are direct applications of the equivalence between the optimization problems $\mathcal{V}$ and $\mathcal{W}$ and their discretized versions and the results of Proposition~\ref{prop:weakconvmeasure}. 
    \end{proof}

\subsection{Uniform Convergence of Grid-Based Solutions} \label{sec:weakconv}

The weak*-convergence results of Theorem~\ref{theo:weak*convergence} have interesting consequences for sampling-admissible operator pairs $(\LT,\LS)$. In this case, sampling evaluations can be used in the sensing process and therefore the weak*-convergence in $\MLTLS$ implies pointwise convergence.
Moreover, this pointwise convergence can be reinforced into a uniform convergence in compact domains under certain conditions~\cite{kelley2017general}. Theorem \ref{theo:uniform} exploits these facts to ensure new uniform convergence results for grid-based methods.
For any locally bounded function $f$ and $a \leq b$, we set $\| f \|_{\infty,[a,b]} = \sup_{a \leq t \leq b} |f(t)|$. 

\begin{theorem}\label{theo:uniform}
    We assume that we are in the conditions of Theorem~\ref{theo:weak*convergence} with the additional assumptions that (i) $(\LT, \LS)$ is sampling admissible and (ii) the operators have locally Lipschitz Green's function and Lipschitz periodic Green's function, respectively.
    
    We fix $(h_{T,n})_{n\geq 0}$ and $(h_{S,n})_{n\geq 0}$ such that $1/ h_{S,n} \in \N$ and $(h_{T,n}, h_{S,n}) \rightarrow_n (0,0)$ and consider a sequence of grid-based solutions $(\tilde{f}_{T,n} , \tilde{f}_{S,n}) \in \mathcal{V}_{h_{T,n},h_{S,n}}$ for each $n \geq 0$. \\
    
    Then, for any accumulation point $(\tilde{f}_{T} , \tilde{f}_{S}) \in \mathcal{V}$ such that $(\tilde{f}_{T,\psi(n)} , \tilde{f}_{S,\psi(n)}) \underset{n}{\overset{*}{\rightarrow}} (\tilde{f}_{T} , \tilde{f}_{S})$ for some subsequence, we have that 
    \begin{equation} \label{eq:this}
    \| \tilde{f}_{T,\psi(n)} - \tilde{f}_{T} \|_{\infty,[a,b]} \underset{n}{\rightarrow} 0
    \quad \text{and} \quad
    \| \tilde{f}_{S,\psi(n)} - \tilde{f}_{S} \|_{\infty} \underset{n}{\rightarrow} 0.
    \end{equation}
    
    If moreover the solution of the continuous-domain optimization problem \eqref{eq:mainopti} is unique and denoted by $(\tilde{f}_{T} , \tilde{f}_{S})$, then we have that, for any finite interval $[a,b] \subset \R$,
    \begin{equation} \label{eq:that}
    \| \tilde{f}_{T,n} - \tilde{f}_{T} \|_{\infty,[a,b]} \underset{n}{\rightarrow} 0
    \quad \text{and} \quad
    \| \tilde{f}_{S,n} - \tilde{f}_{S} \|_{\infty} \underset{n}{\rightarrow} 0.
    \end{equation}
\end{theorem}

\begin{proof}
Let $(\tilde{f}_{T,n} , \tilde{f}_{S,n}) \in \mathcal{V}_{h_{T,n},h_{S,n}}$
and a weak*-convergent subsequence $(\tilde{f}_{T,\psi(n)} , \tilde{f}_{S,\psi(n)})$ whose weak*-limit is denoted by $(\tilde{f}_{T} , \tilde{f}_{S})$. Theorem~\ref{theo:weak*convergence} ensures that such subsequence exists and that its limit is in $\mathcal{V}$. We replace $(\tilde{f}_{T,n} , \tilde{f}_{S,n}) \in \mathcal{V}_{h_{T,n},h_{S,n}}$ by its weak*-converging subsequence (hence removing $\psi(n)$) to simplify the notations. \\
 
\textit{Pointwise convergence.} By assumption, $(\LT,\LS)$ is sampling admissible. In particular, this means that, for any $t_0 \in \R$, the functional 
\begin{equation}
    (f_T,f_S) \mapsto (f_T(t_0), f_S(t_0))
\end{equation}
is weak*-continuous over $\MLTLS$. Then, the weak*-convergence $(\tilde{f}_{T,n} , \tilde{f}_{S,n})$ to $(\tilde{f}_{T} , \tilde{f}_{S})$ implies that, for any $t_0 \in \R$, 
\begin{equation} \label{eq:truc} 
    \tilde{f}_{T,n}(t_0) \underset{n}{\rightarrow} \tilde{f}_{T}(t_0)  
    \quad \text{and} \quad       
    \tilde{f}_{S,n}(t_0) \underset{n}{\rightarrow} \tilde{f}_{S}(t_0)  . 
\end{equation}
In other words, $\left((\tilde{f}_{T,n} , \tilde{f}_{S,n})\right)_{n \geq 0}$ converges pointwise to $(\tilde{f}_{T} , \tilde{f}_{S})$. 

\textit{Lipschitz condition.} 
We show that the functions $\tilde{f}_{T,n}$ are locally Lipschitz with common Lipschitz bound. Let $a \leq b$. For any measurable function $g$, we set $\| g \|_{\mathrm{Lip},[a,b]} = \sup_{a \leq x < y \leq b} \frac{|g(y)-g(x)|}{|y-x|}$.  By assumption, we have that $\| \psi_{\LT} \|_{\mathrm{Lip},[a,b]} < \infty$. 

For any $n \geq 0$, we set
\begin{equation}
            \tilde{f}_{T,n}(t) = \sum_{k \in \Z} a_{k,n} \psi_{\LT} (t - h_{T,n} k) + \sum_{m=1}^{N_T} c_{m,n} t^{m-1} 
\end{equation}
Then, we have for any $a \leq x < y \leq b$, 
\begin{align}
    \frac{|\tilde{f}_{T,n}(y) - \tilde{f}_{T,n}(x) |}{|y-x|}
    & \leq \sum_{k\in\Z} |a_{k,n}| \frac{| \psi_{\LT} (y - h_{T,n} k)-\psi_{\LT} (x - h_{T,n} k)|}{|x-y|} + \sum_{m=1}^{N_T} |c_{m,n}| \frac{|y^{m-1} - x^{m-1} |}{|x-y|}\nonumber \\
    & \leq \| a_n \|_{\ell_1} \| \psi_{\LT} \|_{\mathrm{Lip},[a,b]} + \| \bm{c}_n \|_2    \left(\sum_{m=1}^{N_T} \|(\cdot)^{m-1} \|_{\mathrm{Lip},[a,b]} \right)^{1/2}, \label{eq:citethijustafter}
\end{align}
where we used the Cauchy-Schwartz inequality for the last inequality. 
We moreover observe that $\| a_n \|_{\ell_1} = \| \tilde{w}_{T,n} \|_{\MR}$, which is bounded, as we have seen in \eqref{eq:smallthingtoreuse}. We have also seen in the proof of Theorem~\ref{theo:weak*convergence} that the sequence $\| \bm{c}_n \|_2 $ is bounded. This shows, together with \eqref{eq:citethijustafter}, that 
\begin{equation}
   \| f_{T,n} \|_{\mathrm{Lip},[a,b]} \sup_{a \leq x < y \leq b} = \frac{|\tilde{f}_{T,n}(y) - \tilde{f}_{T,n}(x) |}{|y-x|} \leq C
\end{equation}
for some constant $C > 0$ independent from $n$. We show similarly that the family of functions $\tilde{f}_{S,n}$ are locally Lipschitz with common Lipschitz bound. The functions being $1$-periodic, we deduce that they are Lipschitz over $\R$ (not only locally). \\

\textit{Uniform convergence.} 
The family of functions $(\tilde{f}_{T,n})_{n\geq 0}$ converges pointwise to $\tilde{f}_{T}$. Moreover, the $\tilde{f}_{T,n}$ are locally Lipschitz with common Lipschitz bound, and are therefore equicontinuous over compact domains. The equicontinuity and pointwise convergence implies the uniform convergence~\cite[Theorem 15, Chapter 7]{kelley2017general}, hence we deduce that $\tilde{f}_{T,n}$ converges uniformly to $\tilde{f}_{T}$ on compact domains. Similarly, the sequence $(\tilde{f}_{S,n})_{n\geq 0}$ converges pointwise to $\tilde{f}_{S}$ and the $\tilde{f}_{S,n}$ are  equicontinuous (being Lipschitz with a common bound), hence we deduce that it converges uniformly to $\tilde{f}_{S}$. This proves \eqref{eq:this}, together with \eqref{eq:that} when the solution is unique.      
\end{proof}

\textbf{Remarks.} 
    Theorem~\ref{theo:uniform} generalizes~\cite[Theorem 2]{debarre2022part2}, which only considers (i) a purely periodic setting (single component reconstruction), (ii) Fourier measurements, and (iii) differential operators.
It provides general conditions for the convergence of grid-based methods towards their continuous-domain target problem and is, as far as we know, the first result of that kind with this level of generality.

In the context of seasonal-trend decomposition, it remarkably shows that both the seasonal and trend components are recovered by their respective grid-based approximations. This demonstrates the relevance of grid-based methods for composite signal reconstruction, which is also a novelty to the best of our knowledge.

\section{Conclusion}\label{sec:conclusion}

In this paper, we have developed a variational framework for the reconstruction of 1D continuous-domain functions, decomposed into trend and seasonal components, using generalized total-variation regularization. We introduced a functional-analytic approach to sparsity-promoting reconstruction, particularly suited for inverse problems with noisy linear measurements. Key contributions include the derivation of a representer theorem, the analysis of multi-component TV regularization, and the development of grid-based approximation schemes that converge to the continuous-domain solutions. \\

The seasonal-trend decomposition model enables efficient recovery of both components by combining data fidelity with structural regularization. Our representer theorem characterizes the solutions as sparse splines, and the grid-based methods provide an efficient computational approach with uniform convergence to the continuous-domain solutions. \\

This work lays the foundation for further research in 1D signal decomposition and offers practical methods for real-world applications in time-series analysis and other similar domains where noisy, high-dimensional data are prevalent. The rigorous theory and computational methods presented here make significant strides in advancing the understanding and application of seasonal-trend decomposition.

\bibliographystyle{plain}          
\bibliography{references} 

\end{document}